\documentclass[11pt]{amsart}
\usepackage{amsmath,amsfonts,amsthm,amsopn,amssymb,latexsym,soul}
\usepackage{array}
\usepackage{cite}
\usepackage{color,enumitem,graphicx}
\usepackage[colorlinks=true,urlcolor=blue,
citecolor=red,linkcolor=blue,linktocpage,pdfpagelabels,
bookmarksnumbered,bookmarksopen]{hyperref}
\usepackage[english]{babel}
\usepackage[left=2.61cm,right=2.61cm,top=2.72cm,bottom=2.72cm]{geometry}
\usepackage[hyperpageref]{backref}
\usepackage[colorinlistoftodos]{todonotes}

\usepackage{bm}

\makeatletter
\providecommand\@dotsep{5}
\def\listtodoname{List of Todos}
\def\listoftodos{\@starttoc{tdo}\listtodoname}
\makeatother

\numberwithin{equation}{section}
\newtheorem{theorem}{Theorem}[section]
\newtheorem{lemma}[theorem]{Lemma}

\newtheorem{definition}[theorem]{Definition}

\newtheorem{remark}[theorem]{Remark}
\newtheorem{corollary}[theorem]{Corollary}
\newcommand{\eps}{\varepsilon}
\newcommand{\C}{\mathbb{C}}
\newcommand{\R}{\mathbb{R}}

\newcommand{\RN}{{\mathbb{R}^N}}
\newcommand{\RD}{{\mathbb{R}^2}}
\newcommand{\RT}{{\mathbb{R}^3}}

\renewcommand{\le}{\leqslant}
\renewcommand{\ge}{\geqslant}
\renewcommand{\a }{\alpha }
\newcommand{\dis}{\displaystyle}

\renewcommand{\b }{\beta }

\newcommand{\g }{\gamma }

\renewcommand{\l }{\lambda}
\newcommand{\n }{\nabla }

\renewcommand{\t}{\theta}

\newcommand{\cg}{\mathcal{G}}

\renewcommand{\H}{H^1(\RD)}
\newcommand{\Ha}{H^1_\a(\RD)}

\newcommand{\Hao}{H^1_{\a ,\o}}

\newcommand{\X}{\mathcal{X}}
\newcommand{\N}{\mathbb{N}}

\renewcommand{\C}{\mathbb{C}}
\renewcommand{\o}{\omega}

\def\bbm[#1]{\mbo\X{\boldmath $#1$}}
\newcommand{\beq }{\begin{equation}}
\newcommand{\eeq }{\end{equation}}

\newcommand{\lr}[1]{\langle #1 \rangle}

\allowdisplaybreaks

\title[Coupled nonlinear Schr\"odinger equations with point interaction]{Coupled nonlinear Schr\"odinger equations with point interaction:
existence and asymptotic behaviour}

\author[Y. Osada]{Yuki Osada}
\address{Department of Mathematics, Saitama University,
\newline\indent
Shimo-Okubo 255, Sakura-ku Saitama-shi, 338-8570, Japan}
\email{yukiosada59@gmail.com}

\author[A. Pomponio]{Alessio Pomponio}
\address{Dipartimento di Meccanica, Matematica e Management,
Politecnico di Bari
\newline\indent
Via Orabona 4,  70125  Bari, Italy}
\email{alessio.pomponio@poliba.it}

\thanks{
Y.O. is partially supported by JSPS KAKENHI Grant Number JP23KJ0293.
\\
A.P. is partially supported by  INdAM - GNAMPA Project 2024 ``Metodi variazionali e
topologici per alcune equazioni di Schrodinger nonlineari" CUP E53C23001670001, by INdAM - GNAMPA Project 2025 ``Approcci topologici e variazionali per problemi ellittici in spazi non compatti" CUP E5324001950001.
A.P. is partially financed by European Union - Next Generation EU - PRIN 2022 PNRR ``P2022YFAJH Linear and Nonlinear PDE's: New directions and Applications".}

\begin{document}
\begin{abstract}
In this paper we deal with the following weakly coupled nonlinear Schr\"{o}dinger system
\begin{align*}
\begin{cases}
- \Delta_\alpha u + \omega u = |u|^2 u + \beta u |v|^2&\quad \mathrm{in}\ \mathbb{R}^2,\\
- \Delta v + \tilde{\omega} v = |v|^2 v + \beta |u|^2 v&\quad \mathrm{in}\ \mathbb{R}^2,
\end{cases}
%\tag{$\mathcal{P}_\beta$}
\end{align*}
where $-\Delta_\alpha$ denotes the Laplacian operator with a point interaction, $\o$ greater then a suitable positive constant, $\tilde{\o}>0$,  and $\b\ge 0$. For any $\b\ge 0$ this system admits the existence of a ground state solution which can have only one nontrivial component or two nontrivial components and which could be regular or singular. We analyse this phenomenon showing how this depends strongly on the parameters. Moreover we study the asymptotic behaviour of ground state solutions whenever $\b\to \infty$.
\end{abstract}

\maketitle

\date{\today}
\maketitle

%I commentouted \torange. 
%Moreover, I replaced \newcommand{\vfi}{\varphi} with \renewcommand{\vfi}{\varphi}.
%Because otherwise, it will be an error.

\section{Introduction}

In last years, great attention has been posed to systems of weakly coupled nonlinear Schr\"odinger equations 
of the type
\begin{align}
\label{sist2}
\begin{cases}
- \Delta u + \omega_1 u = |u|^2 u + \beta u |v|^2&\quad \mathrm{in}\ \mathbb{R}^N,\\
- \Delta v + \omega_2 v = |v|^2 v + \beta |u|^2 v&\quad \mathrm{in}\ \mathbb{R}^N,
\end{cases}
\end{align}
where $\o_1,\o_2>0$ and $\b\ge 0$ are fixed parameters.  
Such systems  describe physical phenomena such as the propagation in birefringent optical fibres,
Kerr-like photorefractive media in optics and Bose–Einstein condensates (see, for example, \cite{K,M1,M2}). 
Observe that if $u$ is a nontrivial solution of the single nonlinear Schr\"odinger equation 
\[
-\Delta u + \omega_1 u = |u|^2 u \qquad\text{in }\RN,
\]
then $(u,0)$ is a nontrivial solution of the system \eqref{sist2}. Analogously $(0,v)$ is a nontrivial solution of the system \eqref{sist2} if $v$ is a nontrivial solution of
\[
- \Delta v + \omega_2 v = |v|^2 v \quad \mathrm{in}\ \mathbb{R}^N.
\]
Such types of solutions are usually called \emph{scalar solutions} of \eqref{sist2}, while a solution $(u,v)$ is called \emph{vector} if both the components are nontrivial. So, it is interesting to study the existence of ground state solutions associated to \eqref{sist2} understanding if such solutions are scalar or vector. Actually, this depends on $\b$ and the ground state solution is scalar under a certain value of $\b$ and, after that, it becomes vector. See, for instance,  \cite{LW, MMP, PS}.

A similar phenomenon occurs also for systems of nonlinear Schr\"odinger
equations with three wave interaction of the following type
\begin{align}
\label{sist3}
\begin{cases}
- \Delta u + \omega_1 u = |u|^2 u + \beta vz&\quad \mathrm{in}\ \mathbb{R}^N,\\
- \Delta v + \omega_2 v = |v|^2 v + \beta uz&\quad \mathrm{in}\ \mathbb{R}^N,\\
- \Delta z + \omega_3 z = |z|^2 z + \beta uv&\quad \mathrm{in}\ \mathbb{R}^N,\\
\end{cases}
\end{align}
where, as before, $\b\ge 0$ and $\o_i>0$, for $i=1,2,3$. These systems describe the Raman amplification  used in optical telecommunication (see \cite{HA}). The existence of scalar and vector ground state solutions has been studied in \cite{KurataOsada21CPAA,Pomponio10}. 
In \cite{KurataOsada21CPAA}, moreover, the asymptotic behaviour of ground states solutions whenever $\b$ goes to $\infty$ has been considered.

On the other hand, in the presence of defects
or impurities, the nonlinear Schr\"odinger equation with point interaction has been recently proposed as an effective model for a Bose-Einstein Condensate (BEC). See \cite{SM,SCMS} for the physical background.
In the 1D case,
there has been a lot of works for the nonlinear Schr\"odinger equation with point interaction,
such as the existence of a ground state solution
and the (in)stability of standing waves;
we refer to \cite{BC,FJ, FOO, KO} and references therein.
Instead, the higher dimensional case is less studied. In presence of a pure power nonlinearity, the  
2D problem has been considered  in \cite{ABCT2, FGI22},
while the 3D problem  in \cite{ABCT3}. Instead \cite{PW} deals with a more general nonlinearity in  $2$ and $3$ dimensions. Concerning with time-dependent problems in higher dimensional case,
we refer to \cite{CFN1, CFN2, FN, GMS} and references therein. See also \cite{dPR} for related problems. 

Formally, in the nonlinear Schr\"odinger equation with point interaction, the classical Laplacian operator is replaced by $-\Delta_\alpha$.
Here $-\Delta_\alpha$ can be thought as the two or three dimensional version of the Schr\"odinger operator with the Dirac delta concentrated in zero $-\frac{d^2}{d x^2}+\a \delta_0$.

A rigorous formulation is given through the self-adjoint extension
of the operator $-\Delta|_{C_0^{\infty}(\R^2 \setminus \{0 \})}$.
Then it is known that there exists a family $\{ - \Delta_{\alpha} \}_{\alpha \in \R}$
of self-adjoint operators which realize all point perturbations of $-\Delta$;
see \cite{AGH, AGHH1, AGHH2, AH}.
As a consequence, the domain of $-\Delta_{\alpha}$ is given by 
\begin{align*}
D(-\Delta_\a) &:=
\big\{ u\in L^2(\RD): \text{there exist} \ q \in \C \text{ and }\l >0\text{ s.t. } 
\\&\qquad \qquad
\phi_\l:=u-q\cg_\l\in H^2(\RD) \text{ and }\phi_\l(0)=(\a +\theta_\l)q \big\},
\end{align*}
where $\cg_\l$ is the Green's function of $-\Delta +\l$ on $\R^2$, 
\begin{equation}\label{xi}
\t_\l:=
\dis \frac{\log\left(\frac{\sqrt{\l}}{2}\right)+\g}{2\pi}
\end{equation}
and $\g $ is the Euler-Mascheroni constant.
Moreover the action is defined by
\begin{equation*}
-\Delta_\a u:= -\Delta \phi_\l - q \l \cg_\l, \quad \text{for all } u=\phi_\l+q\cg_\l\in D(-\Delta_\a).
\end{equation*}
It is also known that $\sigma_{ess}(-\Delta_{\alpha})=[0,\infty)$.
Moreover 
$-\Delta_{\alpha}$ has exactly one negative eigenvalue $-\omega_{\alpha}$
which is given by
\begin{equation}\label{omega-al}
\omega_\a:=
4e^{-4\pi \a -2\g }.
\end{equation}
By the definition of $\t_\l$ in \eqref{xi}, 
we find that $\alpha + \t_\l >0$ for any $\l > \omega_{\alpha}$.

The function $u \in D(-\Delta_{\alpha})$ consists of a \textit{regular part} $\phi_\l$,
on which $-\Delta_{\alpha}$ acts as the standard Laplacian, 
and a \textit{singular part} $q \cg_\l$,
on which $-\Delta_{\alpha}$ acts as the multiplication by $-\l$.
These two components are coupled by the boundary condition 
$\phi_\l(0)= (\alpha + \theta_\l)q$.
The strength $q=q(u)$ is called the \textit{charge} of $u$.
In particular, we have that 
\[
\lr{-\Delta_\a u,u}
= \|\n \phi_\l\|_2^2+\l \|\phi_\l \|_2^2 - \lambda \| u\|_2^2
+(\a +\t_\l)|q|^2.
\]
Here $\langle \cdot, \cdot \rangle$ denotes the standard $L^2$-inner product.
As observed in \cite[Remark 2.1]{ABCT2}, $\l$ is a free parameter
and it does not affect the definition of $-\Delta_{\alpha}$ 
nor the charge $q(u)$.

The energy space associated with $-\Delta_{\alpha}$ is defined by 
\begin{equation*}
\Ha:=
\left\{ u\in L^2(\RD): \text{there exist } 
q \in \C \text{ and }\l >0\text{ s.t. } 
\phi_\l:=u-q \cg_\l\in  H^1(\RD) \right\}.
\end{equation*}
We remark that even if we work on this low regularity space,
$q(u)$ is independent of $\l$ and uniquely determined (see Lemma \ref{lem:q} below).
Therefore, in the definition of $\Ha$, 
we do not stress the dependence of $q(u)$ with respect to $\l$.

For any $\o>\omega_\a$, we define the related quadratic form by
\begin{equation*}
\lr{(-\Delta_\a+\o )u,u}=\|\n \phi_\l\|_2^2+\l \|\phi_\l \|_2^2 + (\o - \l) \|u\|_2^2
+(\a +\t_\l)|q|^2,
\end{equation*}
for $u=\phi_\l+q\cg_\l \in \Ha$.
We also put
\begin{equation*} \label{norm}
\|u\|_{\Hao}^2:=\lr{(-\Delta_\a+\o )u,u}
= \|\n \phi_\l\|_2^2+\l \|\phi_\l \|_2^2 
+ (\o - \l) \|u\|_2^2 +(\a +\t_\l)|q|^2.
\end{equation*}
Clearly if $q(u)=0$, then $\|u\|_{\Hao}$ coincides with 
the norm $\|  u \|_{H^1}$.
Moreover it also holds that
\begin{equation} \label{equiv}
\| u\|_{H^1_{\alpha,\omega_1}} \sim \| u \|_{H^1_{\alpha,\o_2}}
\quad \text{for} \ \ \omega_{\alpha} <\o_1 < \o_2.
\end{equation}
See \cite{FGI22} for details.

Motivated by all this, in the present paper we consider the following weakly coupled nonlinear Schr\"{o}dinger system
\begin{align}
\label{eq4}
\begin{cases}
- \Delta_\alpha u + \omega u = |u|^2 u + \beta u |v|^2&\quad \mathrm{in}\ \mathbb{R}^2,\\
- \Delta v + \tilde{\omega} v = |v|^2 v + \beta |u|^2 v&\quad \mathrm{in}\ \mathbb{R}^2,
\end{cases}
\tag{$\mathcal{P}_\beta$}
\end{align}
where $\o>\o_\a$, $\tilde\o>0$, and $\b\ge 0$.

We want to attach this problem by means of variational techniques. So, let us introduce the functional framework. 
We set  $$\mathbb{H}_\a := H^1_\alpha(\mathbb{R}^2) \times H^1(\mathbb{R}^2).$$ Unless otherwise stated, we fix $\omega_\alpha < \lambda \le \omega$. Solutions of \eqref{eq4} will be found as critical points of the functional $I_\beta : \mathbb{H}_\a \to \mathbb{R}$, where
\begin{align*}
I_\beta(u,v) &:= \frac{1}{2} \left(
\langle (- \Delta_\alpha + \omega)u,u \rangle + \|\nabla v\|_2^2 + \tilde{\omega} \|v\|_2^2
\right) - \frac{1}{4} (\|u\|_4^4 + \|v\|_4^4) - \frac{\beta}{2} \int_{\mathbb{R}^2} |uv|^2\\
&= \frac{1}{2} \left(\|\nabla \phi_\lambda\|_2^2 + \lambda \|\phi_\lambda\|_2^2 + (\omega - \lambda) \|u\|_2^2 + (\alpha + \theta_\lambda) |q|^2\right)\\
&\quad \ + \frac{1}{2} \left(\|\nabla v\|_2^2 + \tilde{\omega} \|v\|_2^2\right) - \frac{1}{4} (\|u\|_4^4 + \|v\|_4^4) - \frac{\beta}{2} \int_{\mathbb{R}^2} |uv|^2,
\end{align*}
for any $u=\phi_\lambda+q\cg_\l\in \Ha$ and $v\in \H$. 

We will look for ground state solutions and so we need to define the ground state level
\[
c_\beta := c_\beta (\o,\tilde{\omega}):= \inf_{(u,v) \in \mathcal{N}_\beta} I_\beta(u,v),
\]
where $\mathcal{N}_\beta$ is the Nehari manifold
\begin{align*}
\mathcal{N}_\beta := \big\{(u,v) \in \mathbb{H}_\a \setminus \{(0,0)\} \mid G_\beta(u,v) = 0\big\},
\end{align*}
with
\begin{align*}
G_\beta(u,v)& := \langle (- \Delta_\alpha + \omega)u,u \rangle + \|\nabla v\|_2^2 + \tilde{\omega} \|v\|_2^2 - (\|u\|_4^4 + \|v\|_4^4) - 2 \beta \int_{\mathbb{R}^2} |uv|^2\\
&= \|\nabla \phi_\lambda\|_2^2 + \lambda \|\phi_\lambda\|_2^2 + (\omega - \lambda) \|u\|_2^2 + (\alpha + \theta_\lambda) |q|^2\\
&\quad \ + \|\nabla v\|_2^2 + \tilde{\omega} \|v\|_2^2 - (\|u\|_4^4 + \|v\|_4^4) - 2 \beta \int_{\mathbb{R}^2} |uv|^2.
\end{align*}
As observed, for example, in \cite{ABCT2}, $I_\b$ \emph{does not depend} on the choice of $\l$ and so also the Nehari manifold $\mathcal{N}_\beta$ and the ground state level $c_\b$. 

We will need to compare the functional $I_\b$ with the corresponding one in absence of the point interaction. Namely we introduce $I_\beta^0 : \mathbb{H}\to \mathbb{R}$ where, for any $(u,v)\in \mathbb{H} := H^1(\RD) \times H^1(\RD)$, 
\begin{align*}
I_\beta^0(u,v)& := \frac{1}{2} \left(
\|\nabla u\|_2^2 + \omega \|u\|_2^2 + \|\nabla v\|_2^2 + \tilde{\omega} \|v\|_2^2
\right) - \frac{1}{4} (\|u\|_4^4 + \|v\|_4^4) - \frac{\beta}{2} \int_{\mathbb{R}^2} |uv|^2.
\end{align*}

Analogously as before, we define its associated ground state level
\begin{align*}
c_\beta^0 := c_\beta^0(\omega,\tilde{\omega}) := \inf_{(u,v) \in \mathcal{N}_\beta^0} I_\beta^0(u,v),
\end{align*}
where $\mathcal{N}_\beta^0$ is the Nehari manifold
\begin{align*}
\mathcal{N}_\beta^0 := \big\{(u,v) \in \mathbb{H} \setminus \{(0,0)\} \mid G_\beta^0(u,v) = 0\big\},
\end{align*}
with
\begin{align*}
G_\beta^0(u,v) := \|\nabla u\|_2^2 + \omega \|u\|_2^2 + \|\nabla v\|_2^2 + \tilde{\omega} \|v\|_2^2 - (\|u\|_4^4 + \|v\|_4^4) - 2 \beta \int_{\mathbb{R}^2} |uv|^2.
\end{align*}
It's well known that the ground state for $c_\b^0$ is attained. See, for example, \cite{AC,LW,PS} and references therein.

Analogously, we have to introduce the corresponding functionals, Nehari manifolds and ground states levels for the single nonlinear Schr\"odinger equation both in presence of the point interaction and without the point interaction. More precisely we set $S_\omega : H^1_\alpha(\mathbb{R}^2) \to \mathbb{R}$ where 
\begin{align*}
S_\omega(u)& := \frac{1}{2} 
\langle (- \Delta_\alpha + \omega)u,u \rangle - \frac{1}{4} \|u\|_4^4\\
&= \frac{1}{2} \left(\|\nabla \phi_\lambda\|_2^2 + \lambda \|\phi_\lambda\|_2^2 + (\omega - \lambda) \|u\|_2^2 + (\alpha + \theta_\lambda) |q|^2\right) - \frac{1}{4} \|u\|_4^4,
\end{align*}
for any $u=\phi_\lambda+q\cg_\l\in \Ha$. Moreover we denote
\begin{align*}
\mathcal{M}_\omega := \big\{u \in H^1_\alpha(\mathbb{R}^2) \setminus \{0\} \mid \langle (- \Delta_\alpha + \omega)u,u \rangle = \|u\|_4^4\big\},
\end{align*}
the associated Nehari manifold and
\begin{align*}
d(\omega) := \inf_{u \in \mathcal{M}_\omega} S_\omega(u),
\end{align*}
the ground state level.
As shown in \cite[Theorem 1.11]{ABCT2}, $d(\o)$ is attained and any minimizer is singular. 

Finally, we set $S_{\tilde{\omega}}^0 : H^1(\mathbb{R}^2) \to \mathbb{R}$ such that, for any $v\in H^1(\mathbb{R}^2)$, 
\begin{align*}
S_{\tilde{\omega}}^0(v) := 
 \frac{1}{2} \left(\|\nabla v\|_2^2 + \tilde{\omega} \|v\|_2^2\right) - \frac{1}{4} \|v\|_4^4,
\end{align*}
and
\begin{align*}
d^0(\tilde{\omega}) := \inf_{v \in \mathcal{M}_{\tilde{\omega}}^0} S_{\tilde{\omega}}^0(v),
\end{align*}
with
\begin{align*}\mathcal{M}_{\tilde{\omega}}^0 := \big\{v \in H^1(\mathbb{R}^2) \setminus \{0\} \mid \|\nabla v\|_2^2 + \tilde{\omega} \|v\|_2^2 = \|v\|_4^4\big\}.
\end{align*}
It's well known that $d^0(\tilde{\omega}) $ is attained.

In the sequel we need also to introduce 
\begin{align}
\label{b0}
\beta_0 &:= \beta_0(\o,\tilde{\omega}) := \max \{\beta \ge 0 \mid c_\beta^0 = c_0^0\},
\\
\label{b*}
\beta^* &:= \beta^*(\o,\tilde{\omega}) := \max \{\beta \ge 0 \mid c_\beta = c_0\}.
\end{align}
Observe that one can show that $\b_0>0$ and it is attained arguing as in \cite{KurataOsada21CPAA}. Also $\b^*$ is positive and attained (see Lemma \ref{le:b*} and Lemma \ref{lemm13}).

\

Our first existence result is the following.

\begin{theorem}\label{th:ex}
Assume that $\alpha \in \mathbb{R}$, $\omega > \omega_\alpha$, $\tilde{\omega} > 0$ 
and $\beta \ge 0$. Then there exists a ground state $(u_0,v_0) \in \mathbb{H}_\alpha$ of \eqref{eq4}. 
\end{theorem}

Our main aim in this paper is not only the existence of a ground state solution, but also the nature of such solution. 

First of all, observe that if $u\in \Ha$ is a nontrivial solution of 
\begin{equation*}
- \Delta_\alpha u + \omega u = |u|^2 u \quad \mathrm{in}\ \mathbb{R}^2,
\end{equation*}
then $(u,0)$ is a nontrivial solution of \eqref{eq4}. Analogously, if $v\in \H$ is a nontrivial solution of 
\begin{equation*}
- \Delta v + \tilde{\omega} v = |v|^2 v \quad \mathrm{in}\ \mathbb{R}^2,
\end{equation*}
then $(0,v)$ is a nontrivial solution of \eqref{eq4}, too. 

In order to distinguish these \emph{semi-trivial} solutions from those with both the two components different from zero, we need the following definition.

\begin{definition}
Let $(u,v) \in \mathbb{H}_\alpha$ be a nontrivial solution of \eqref{eq4}. We say that it is a \emph{scalar solution} if $u=0$ or $v=0$. Conversely, we say that it is a \emph{vector solution} if $u\neq0$ and $v\neq0$. 
\end{definition}

Moreover, we would like to investigate if the first component of the solutions contains or not a singular part.

\begin{definition}
Let $(u,v) \in \mathbb{H}_\alpha$. Then there exist $\phi_\lambda \in H^1(\mathbb{R}^2)$ and $q \in \mathbb{C}$ such that $u = \phi_\lambda + q \cg_\l $. 
We say that $(u,v)$ is regular if $q = 0$. While we say that $(u,v)$ is singular if $q \neq 0$. 
\end{definition}

We will see how all these properties of the solutions are strictly related to the parameter $\b$ and to all the mountain pass levels introduced before. Therefore we need to study carefully all the possible different scenarios. Actually the situation is richer and more delicate in comparison with problems without a point interaction such as \eqref{sist2} and \eqref{sist3}.

We start with the following result.

\begin{theorem} \label{theo1}
Assume that $\alpha \in \mathbb{R}$, $\omega > \omega_\alpha$, $\tilde{\omega} > 0$ 
and $\beta \ge 0$. 
Then if $c_\b<c_\b^0$, then there exists a ground state $(u_0,v_0) \in \mathbb{H}_\alpha$ of \eqref{eq4} such that $q_0 \neq 0$, where $q_0$ is the charge of $u_0$.
Namely, $c_\beta$ has a singular minimizer. 
Moreover, $\phi_{\lambda,0} \neq 0$ holds, where $u_0 = \phi_{\lambda,0} + q_0 \cg_\l $. 
\end{theorem}

From the proof of Theorem \ref{theo1}, we find that $(u,v)$ is a minimizer for $c_\b$ if and only if $(u,v)$ is a ground state of \eqref{eq4}. 
So in what follows, we use the terms ``ground state" and ``minimizer" without any distinction.

In Lemma \ref{lemm2} and Lemma \ref{lemm3}, we give some conditions which assure that $c_\b<c_\b^0$ and so we have the following.

\begin{corollary}\label{cor}
Assume that $\alpha \in \mathbb{R}$, $\omega > \omega_\alpha$, $\tilde{\omega} > 0$ 
and $\beta > 0$. Then if we assume either $\omega \le \tilde{\omega}$ or $\b > \b_0$, then there exists a ground state $(u_0,v_0) \in \mathbb{H}_\alpha$ of \eqref{eq4} such that $q_0 \neq 0$, where $q_0$ is the charge of $u_0$.
Namely, $c_\beta$ has a singular minimizer. 
Moreover, $\phi_{\lambda,0} \neq 0$ holds, where $u_0 = \phi_{\lambda,0} + q_0 \cg_\l $. 
\end{corollary}

Now we show that for \emph{small} value of $\b$, the ground state is scalar and, in this case, the parameters $\o, \tilde{\o}$ play a crucial role. More precisely, the following holds.

\begin{theorem} \label{theo3}
If $0 \le \beta <\beta^*$  then for any minimizer $(u_0,v_0)$ for $c_\beta$ is scalar. If, in addition, $d(\omega) > d^0(\tilde{\omega})$, then $u_0 = 0$ and so $(u_0,v_0)$ is scalar and regular; while if $d(\omega) < d^0(\tilde{\omega})$, then $v_0 = 0$ and $(u_0,v_0)$ is scalar and singular.
If $d(\omega) = d^0(\tilde{\omega})$, then it follows that $u_0 = 0$ or $v_0 = 0$. 
\end{theorem}

\begin{remark} \label{rem6}
By the scaling argument, $d^0(\tilde{\omega}) = d^0(1) \tilde{\omega}$ holds. Therefore, the followings hold:
\begin{itemize}
\item[(i)] if $\tilde{\omega} > d(\omega) / d^0(1)$, then $d(\omega) < d^0(\tilde{\omega})$ holds; 

\item[(ii)] if $\tilde{\omega} = d(\omega) / d^0(1)$, then $d(\omega) = d^0(\tilde{\omega})$ holds;

\item[(iii)] if $\tilde{\omega} < d(\omega) / d^0(1)$, then $d(\omega) > d^0(\tilde{\omega})$ holds. 
\end{itemize}
\end{remark}

The next result shows that \eqref{eq4} does not admit the existence of vector and regular ground state.

\begin{theorem} \label{theo4}
For any $\b \ge 0$, there is no vector and regular minimizer for $c_\beta$. 
\end{theorem}

\begin{remark} \label{rem5}
From Theorem \ref{theo1} and Theorem \ref{theo3}, it follows that
if $0 \le \beta < \beta^*$ and $c_\beta < c_\beta^0$ hold, then all minimizers for $c_\beta$ are scalar and singular. 
\end{remark}

Our aim is to prove that for $\b$ sufficiently large all the ground states are vector and singular. In order to do this, first we need the following results.

\begin{theorem} \label{theo5}
Let $\beta \ge 0$ and assume that $c_\beta = c_\beta^0$. Then $$\omega > \tilde{\omega}, \qquad\b\le\beta^* \le \beta_0, \qquad
c_\beta=c_\beta^0 = c_0^0=c_0=d^0(\tilde{\omega})\le d(\o).$$ 
Moreover, we have
\begin{align*}
&c_{\beta'} = c_{\beta'}^0\quad \text{for any}\ 0 \le \beta' \le \beta^*,\\
&c_{\beta'} < c_{\beta'}^0\quad \text{for any}\ \beta' > \beta^*.
\end{align*}

\end{theorem}

\begin{remark}\label{remm}
By Theorem \ref{theo5} we deduce that either $c_\b<c_\b^0$, for any $\b\ge 0$, or $c_{\beta} = c_{\beta}^0$, for any $0 \le \beta \le \beta^*$, and $c_{\beta} < c_{\beta}^0$, for any $\beta > \beta^*$. In any case, we have that $c_{\beta} < c_{\beta}^0$, for any $\beta > \beta^*$.
\end{remark}

\begin{theorem} \label{theo7}
Let $\b \ge 0$. Then the followings are equivalent:
\begin{enumerate}[label=(\roman{*}),ref=\roman{*}]
\item \label{7i}$c_\beta = c_\beta^0$; 
\item \label{7ii} $d(\omega) \ge d^0(\tilde{\omega})$ and $\beta \le \beta^*$.
\end{enumerate}
\end{theorem}

We can easily find that the following corollary is correct by contraposition.

\begin{corollary} \label{coro8}
Let $\b \ge 0$. Then the followings are equivalent:
\begin{enumerate}[label=(\roman{*}),ref=\roman{*}]
\item \label{8i} $c_\beta < c_\beta^0$; 
\item \label{8ii} $d(\omega) < d^0(\tilde{\omega})$ or $\beta > \beta^*$.
\end{enumerate}
\end{corollary}

Finally we can say that for \emph{large} $\b$ the ground state solution of \eqref{eq4} is always vector and singular.

\begin{theorem} \label{theo6}
If $\beta > \beta^*$, all minimizers for $c_\beta$ are vector and singular. 
\end{theorem}

\begin{remark} \label{rem4}
Assume $c_\beta = c_\beta^0$ and so $0 \le \beta <\beta^*$.
In this case, $d(\omega) \ge d^0(\tilde{\omega})$ holds by Theorem \ref{theo5}.  Then from Theorem \ref{theo3}, the followings hold.
\begin{itemize}
\item[(i)] If $d(\omega) = d^0(\tilde{\omega})$, then
$\mathcal{E}_\beta = (\mathcal{E}(\omega) \times \{0\}) \cup (\{0\} \times \mathcal{E}^0(\tilde{\omega}))$, 
where $\mathcal{E}_\beta$ is the set of minimizers for $c_\beta$, $\mathcal{E}(\omega)$ is the set of minimizers for $d(\omega)$ and $\mathcal{E}^0(\tilde{\omega})$ is the set of minimizers for $d^0(\tilde{\omega})$.
Namely, there exists a singular minimizer for $c_\beta$ when $c_\beta = c_\beta^0$. 
%$(u_0,v_0)$ is a minimizer for $c_\beta$ if and only if $(u_0,v_0)$ is a minimizer for $d(\omega)$ or $d^0(\tilde{\omega})$. 

\item[(ii)] If $d(\omega) > d^0(\tilde{\omega})$, then $\mathcal{E}_\beta = \{0\} \times \mathcal{E}^0(\tilde{\omega})$. 
\end{itemize}
\end{remark}

Summarizing we have the following.
\begin{remark}
Let $\b\ge 0$ and $(u_0,v_0)\in \mathbb{H}_\alpha$ be a ground state  of \eqref{eq4}. Then we have:
\begin{center}
\begin{tabular}{|m{2,2cm}|m{4,5cm}|m{4,5cm}|}
\hline
&$0\le \b <\b^*$&$\b>\b^*$\\
\hline
$d(\o)>d^0(\tilde\o)$&scalar and regular ground state with $u_0=0$
&vector and singular ground state \\
\hline
$d(\o)=d^0(\tilde\o)$&scalar and regular ground state with $u_0=0$ or scalar and singular ground state with $v_0=0$
&vector and singular ground state  \\
\hline
$d(\o)<d^0(\tilde\o)$&scalar and singular ground state with $v_0=0$
& vector and singular ground state \\
\hline
\end{tabular}
\end{center}

\end{remark}

All the previous theorems will be proved in Section \ref{se:gs}, after some preliminaries, together with some useful observations.

\

The second part of the paper, instead, is devoted to the study of the asymptotic behaviour of ground state solutions of \eqref{eq4} as $\beta \to \infty$. To this aim, we introduce the following limit problem
\begin{align}
\label{eq11}
\begin{cases}
- \Delta_\alpha u + \omega u = u |v|^2& \mathrm{in}\ \mathbb{R}^2,\\
- \Delta v + \tilde{\omega} v = |u|^2 v& \mathrm{in}\ \mathbb{R}^2.
\end{cases}
\tag{$\mathcal{\tilde{P}}_\infty$}
\end{align}
The associated  functional $\tilde{I}_\infty : \mathbb{H}_\alpha \to \mathbb{R}$ is so defined
\[
\tilde{I}_\infty(u,v) := \frac{1}{2} \left( \langle (- \Delta_\alpha + \omega)u,u \rangle + \|\nabla v\|_2^2 + \tilde{\omega} \|v\|_2^2\right) - \frac{1}{2} \int_{\mathbb{R}^2} |uv|^2,
\]
and its ground state level is 
\begin{equation}\label{c_inf}
\tilde{c}_\infty := \inf_{(u,v) \in \mathcal{\tilde{N}}_\infty}\tilde{I}_\infty(u,v),
\end{equation}
where 
\[
\mathcal{\tilde{N}}_\infty := \{(u,v) \in \mathbb{H}_\a \setminus \{(0,0)\} \mid \tilde{G}_\infty(u,v) = 0\},
\]
with 
\[\tilde{G}_\infty(u,v) := 
\langle (- \Delta_\alpha + \omega)u,u \rangle + \|\nabla v\|_2^2 + \tilde{\omega} \|v\|_2^2 - 2 \int_{\mathbb{R}^2} |uv|^2.
\]

We have the following asymptotic result.
\begin{theorem} \label{theo2}
It follows that $c_\beta = \tilde{c}_\infty / \beta + o(1/\beta)$ as $\beta \to \infty$. 
Moreover, let $\b_n \to \infty$ with $\b_n > 0$ and $(u_n,v_n)$ be a ground state of $(\mathcal{P}_{\beta_n})$ with $u_n = \phi_{\lambda,n} + q_n \cg_\l \ (\phi_{\lambda,n} \in H^1(\mathbb{R}^2),\ q_n \in \mathbb{C},\ \omega_\alpha < \lambda \le \omega)$. Then up to a subsequence, there exists a ground state $(w_0,z_0) \in \mathbb{H}_\alpha\ (w_0 = \psi_{\lambda,0} + r_0 \cg_\l ,\ \psi_{\lambda,0} \in H^1(\mathbb{R}^2),\ r_0 \in \mathbb{C} \setminus \{0\})$ of \eqref{eq11} such that
\begin{align*}
&\sqrt{\beta_n} \phi_{\lambda,n} \to \psi_{\lambda,0}\quad \mathrm{in}\ H^1(\mathbb{R}^2),\\
&\sqrt{\beta_n} q_n \to r_0\quad \mathrm{in}\ \mathbb{C},\\
&\sqrt{\beta_n} v_n \to z_0\quad \mathrm{in}\ H^1(\mathbb{R}^2). 
\end{align*}
\end{theorem}

The proof of the previous result can be found in Section \ref{b^*:to:inf}.

We conclude the introduction with some observations.
\\
While the nonlinear Schr\"odinger equation with point interaction has been studied up to dimension~$3$, we cannot treat system \eqref{eq4} in $\RT$ with our variational techniques because $\cg_\l\in L^\tau(\RT)$, only for $\tau\in [1,3)$, and so the functional $I_\b$ would not be well defined in this case.
\\
In addition to \eqref{eq4}, it is possible to consider a system which contains two  point interactions, namely
\begin{align}
\label{eqaa}
\begin{cases}
- \Delta_\alpha u + \omega u = |u|^2 u + \beta u |v|^2&\quad \mathrm{in}\ \mathbb{R}^2,\\
- \Delta_\a v + \tilde{\omega} v = |v|^2 v + \beta |u|^2 v&\quad \mathrm{in}\ \mathbb{R}^2,
\end{cases}
\end{align}
where $\o>\o_\a$, $\tilde\o>\o_\a$, and $\b\ge 0$. The analogous of our results holds also for \eqref{eqaa}. However, in our opinion, the simultaneous presence of two different differential operators in \eqref{eq4} makes this problem more intriguing then \eqref{eqaa}.

\section{Ground state solutions and their properties}\label{se:gs}

%$\cg_\lambda$ is the Green function of $- \Delta + \lambda$ in $\RD$, defined by the distributional relation $(- \Delta + \lambda) \cg_\lambda = \delta$, that is, 
%\begin{align*}
%\cg_\lambda(x) = \frac{1}{2 \pi} \mathcal{F}^{-1}\left[\frac{1}{|\xi|^2 + \lambda}\right](x) = \frac{1}{(2 \pi)^2} \int_{\RD} \frac{e^{- x \cdot \xi}}{|\xi|^2 + \lambda} d \xi,
%\end{align*}
%where $\mathcal{F}^{-1}$ denotes inverse Fourier transform.
%%

This section is devoted to the proofs of our main theorems about the existence and the properties of ground states solutions of \eqref{eq4}. 

We start recalling some useful properties of $\cg_\lambda$.
We have that for all $\eps > 0$ and $\lambda,\mu > 0$, 
\begin{align*}
&\cg_\lambda \in H^{1-\eps}(\RD),\quad \cg_\lambda \not \in H^1(\RD),\\
&\cg_\lambda - \cg_\mu \in H^{3 - \eps}(\RD),\quad \cg_\lambda - \cg_\mu \not \in H^3(\RD)\quad \text{if}\ \lambda \neq \mu.
\end{align*}
We have also that $\cg_\l\in L^\tau(\RD)$, for any $\tau\in [2,\infty)$.
Moreover, $\cg_\lambda$ is positive, radial, and strictly decreasing (see \cite{ABCT2,FGI22}).

As observed in \cite{ABCT2}, the charge does not depend on $\l$.  Indeed, the following holds.
\begin{lemma} \label{lem:q}
Let $u \in \Ha$ be given. 
Then $q=q(u)$ does not depend on the choice of $\l >0$ 
and so it is determined uniquely.
\end{lemma}

In \cite{FGI22}, Fukaya-Georgiev-Ikeda stated the following lemma which guarantees that the functional $I_\b$ is well defined.

\begin{lemma} \label{lemm20}
For any $\tau \in [2,\infty)$ and $\o > \omega_\alpha$, there exists a constant $C > 0$ such that
\begin{align*}
\|u\|_\tau \le C \|u\|_{H^1_{\alpha,\o}}
\end{align*}
for all $u \in H^1_\alpha(\mathbb{R}^2)$. 
\end{lemma}

If not specified otherwise, we will always take $\o > \omega_\alpha$, $\tilde{\omega}>0$, and $\b\ge 0$.
\begin{lemma} \label{lemm14}
There exists $C > 0$ such that
\begin{align*}
\|(u,v)\|_{\mathbb{H}_\alpha}^2 := \langle (- \Delta_\alpha + \omega)u,u \rangle + \|\nabla v\|_2^2 + \tilde{\omega} \|v\|_2^2 \ge C\quad \mathrm{for\ all}\ (u,v) \in \mathcal{N}_\beta.
\end{align*}
\end{lemma}

\begin{proof}
From $G_\b(u,v) = 0$ and  \cite[Lemma 2.1]{FGI22}, we have
\begin{align*}
\|(u,v)\|_{\mathbb{H}_\alpha}^2 
%\langle (- \Delta_\alpha + \omega)u,u \rangle + \|\nabla v\|_2^2 + \tilde{\omega} \|v\|_2^2
 = \|u\|_4^4 + \|v\|_4^4 + 2 \beta \int_{\mathbb{R}^2} |uv|^2
 \le C (\|u\|_{H^1_{\alpha,\o}}^4 + \|v\|_{H^1}^4)
 \le C \|(u,v)\|_{\mathbb{H}_\alpha}^4.
\end{align*}
\end{proof}

\begin{lemma} \label{lemm15}
For all $(u,v) \in \mathcal{N}_\beta$, $G_\beta'(u,v)[(u,v)] < 0$. 
\end{lemma}

\begin{proof}
Observe that
\begin{align*}
G_\b(u,v)[(u,v)] &= \left. \frac{d}{dt} G_\b(tu,tv) \right|_{t=1}
\\&= 2 \|(u,v)\|_{\mathbb{H}_\alpha}^2 - 4(\|u\|_4^4 + \|v\|_4^4) - 8 \beta \int_{\RD} |uv|^2
\\&= - 2 \|(u,v)\|_{\mathbb{H}_\alpha}^2 < 0.
\end{align*}
\end{proof}

\begin{lemma}\label{min-crit}
If $(u_0,v_0)\in \mathbb{H}_\a$ is a minimizer for $c_\beta$, then $I_\beta'(u_0,v_0) = 0$. 
\end{lemma}

\begin{proof}
Let $(u_0,v_0)\in  \mathbb{H}_\a$ be a minimizer for $c_\beta$. Then there exists a Lagrange multiplier $\eta \in \mathbb{R}$ such that $I_\beta'(u_0,v_0) = \eta G_\beta'(u_0,v_0)$. Moreover, it follows that
\begin{align*}
0 = I_\beta'(u_0,v_0)[(u_0,v_0)] = \eta G_\beta'(u_0,v_0)[(u_0,v_0)]. 
\end{align*}
From Lemma \ref{lemm15}, we have $\eta = 0$ and hence $I_\beta'(u_0,v_0) = 0$. 
\end{proof}

\begin{lemma} \label{lemm16}
Let $(u,v) \in \mathbb{H}_\alpha \setminus \{(0,0)\}$. Then the function $t \mapsto I_\b(tu,tv)$ has a unique maximum point $t_0$ on $(0,\infty)$ and $(t_0 u, t_0 v) \in \mathcal{N}_\beta$. 
\end{lemma}

\begin{proof}
Let $f(t) := I_\beta(tu,tv) = \frac{A}{2} t^2 - \frac{B}{4} t^4 - \frac{\beta}{2} C t^4$, where
\begin{align*}
A = \|(u,v)\|_{\mathbb{H}_\alpha}^2 > 0,\quad
B = \|u\|_4^4 + \|v\|_4^4 > 0,\quad 
C = \int_{\RD} |uv|^2 \ge 0.
\end{align*}
Then $f'(t) = A t - B t^3 - 2 \beta C t^3$. If we assume $f'(t) = 0$ and $t > 0$, then $t^2 = t_0^2 := \frac{A}{B + 2 \beta C} > 0$. Moreover, since $f''(t_0) = A - 3 B t_0^2 - 6 \beta C t_0^2 = -2 A < 0$, $t_0$ is a unique maximum point of the function $f(t)$. 
\end{proof}

Arguing as before, we have the following.
\begin{lemma} \label{lemm16-bis}
Let $(u,v) \in \mathbb{H} \setminus \{(0,0)\}$. Then the function $t \mapsto I_\b^0(tu,tv)$ has a unique maximum point $t_0$ on $(0,\infty)$ and $(t_0 u, t_0 v) \in \mathcal{N}_\beta^0$. 
\end{lemma}

\begin{lemma} \label{lemm17}
Let $(u,v)\in \mathbb{H}_\a\setminus\{(0,0)\}$. If $G_\beta(u,v) < 0$, then $J(u,v) > c_\beta$, where
\begin{align*}
J(u,v) := \frac{1}{4} \|(u,v)\|_{\mathbb{H}_\alpha}^2.
\end{align*}
\end{lemma}

\begin{proof}
From Lemma \ref{lemm16}, there exists $t_0 > 0$ such that $I_\beta(t_0 u,t_0 v) = \max_{t > 0} I_\beta(t u,t v)$ and $(t_0 u,t_0 v) \in \mathcal{N}_\beta$. Since $G_\beta(t_0 u,t_0 v) = 0$ and $G_\beta(u,v) < 0$, it follows that $t_0 < 1$. Therefore
\begin{align*}
c_\beta \le I_\beta(t_0 u,t_0 v) = J(t_0 u,t_0 v) = t_0^2 J(u,v) < J(u,v).
\end{align*}
\end{proof}

We are now ready to prove Theorem \ref{theo1}.

\begin{proof}[\textbf{Proof of Theorem \ref{theo1}}]
Fix  $\omega_\alpha < \lambda \le \omega$. Let $\{(u_n,v_n)\}_{n} \subset \mathcal{N}_\beta$ be a minimizing sequence for $c_\beta$, namely, $I_\beta(u_n,v_n) \to c_\beta$. Then
\begin{align*}
c_\beta+o(1) = I_\beta(u_n,v_n) = J(u_n,v_n)
 = \frac{1}{4} \|(u_n,v_n)\|_{\mathbb{H}_\alpha}^2 \ge C(\|\phi_{\lambda,n}\|_{H^1}^2 + |q_n|^2 + \|v_n\|_{H^1}^2),
\end{align*}
where $u_n = \phi_{\lambda,n} + q_n \cg_\l $. Therefore $\{\phi_{\lambda,n}\}_n$ and $\{v_n\}_n$ are bounded in $H^1(\RD)$ and $\{q_n\}_n$ is bounded in $\C$. Thus, up to a subsequence, there exist $\phi_{\lambda,0},v_0 \in H^1(\RD)$ and $q_0 \in \C$ such that
\begin{align*}
&\phi_{\lambda,n} \rightharpoonup \phi_{\lambda,0}\ \mathrm{weakly\ in}\ H^1(\RD),\\
&q_n \to q_0\ \mathrm{in}\ \C,\\
&v_n \rightharpoonup v_0\ \mathrm{weakly\ in}\ H^1(\RD).
\end{align*}
%We will prove that  $q_0 \neq 0$. 

\textbf{Claim 1.}\quad Up to a subsequence, there exists $C > 0$ such that $\|u_n\|_4^4 + \|v_n\|_4^4 \ge C$ for all $n \in \N$. 
\\
Assume, by contradiction, that $\|u_n\|_4^4 + \|v_n\|_4^4 \to 0$. Then $\int_\RD |u_n v_n|^2 \to 0$. Since $(u_n,v_n) \in \mathcal{N}_\beta$, we have
\begin{multline*}
\|\nabla \phi_{\lambda,n}\|_2^2 + \lambda \|\phi_{\lambda,n}\|_2^2 + (\omega - \lambda) \|u_n\|_2^2 + (\alpha + \theta_\lambda) |q_n|^2 + \|\nabla v_n\|_2^2 + \tilde{\omega} \|v_n\|_2^2
\\
= \|u_n\|_4^4+\|v_n\|_4^4 + 2 \beta \int_\RD |u_n v_n|^2
\end{multline*}
and so $\|\phi_{\lambda,n}\|_{H^1} \to 0$, $|q_n| \to 0$ and $\|v_n\|_{H^1} \to 0$. On the other hand, from Lemma \ref{lemm14}, there exists $C > 0$ such that $\|\phi_{\lambda,n}\|_{H^1}^2 + |q_n|^2 + \|v_n\|_{H^1}^2 \ge C$. This is a contradiction, proving the claim. 

\textbf{Claim 2.}\quad $q_0 \neq 0$. 
\\
Assume $q_0 = 0$. Then 
\begin{align*}
\|u_n - \phi_{\lambda,n}\|_p &\to 0\quad \mathrm{for}\ 2 \le p< \infty,\\ 
\left| \|u_n\|_p^p - \|\phi_{\lambda,n}\|_p^p\right| &\to 0\quad \mathrm{for}\ 2 \le p< \infty,\\
\left| \int_\RD |u_n v_n|^2 - \int_\RD |\phi_{\lambda,n} v_n|^2 \right| &\to 0. 
\end{align*}
%Since $\|u_n\|_4^4 + \|v_n\|_4^4 \ge C$ for all $n \in \N$, 
By Claim 1 we have, up to a subsequence, 
\begin{align*}
\|\phi_{\lambda,n}\|_4^4 + \|v_n\|_4^4 \ge \frac{C}{2}\quad \mathrm{for\ all}\ n \in \N.
\end{align*}
From Lemma \ref{lemm16-bis}, there exists $t_n > 0$ such that $(t_n \phi_{\lambda,n},t_n v_n) \in \mathcal{N}_\beta^0$ and so
\begin{align*}
\|\nabla \phi_{\lambda,n}\|_2^2 + \omega \|\phi_{\lambda,n}\|_2^2 + \|\nabla v_n\|_2^2 + \tilde{\omega} \|v_n\|_2^2 = t_n^2 \left(\|\phi_{\lambda,n}\|_4^4 + \|v_n\|_4^4 + 2 \beta \int_\RD |\phi_{\lambda,n} v_n|^2\right).
\end{align*}
This implies that $\{t_n\}$ is bounded and, since
\begin{align*}
c_\beta + o(1) = I_\beta(u_n,v_n) \ge I_\beta(t_n u_n, t_n v_n)
% &= \frac{t_n^2}{2} \{\|\nabla \phi_{\lambda,n}\|_2^2 + \lambda \|\phi_{\lambda,n}\|_2^2 + (\omega - \lambda) \|u_n\|_2^2 + (\alpha + \theta_\lambda) |q_n|^2 + \|\nabla v_n\|_2^2 + \tilde{\omega} \|v_n\|_2^2
%\}\\
%&\quad \ - t_n^4 \left(\frac{1}{4} (\|u_n\|_4^4 + \|v_n\|_4^4) + \frac{\beta}{2} \int_\RD |u_n v_n|^2\right)
 = I_\beta^0(t_n \phi_{\lambda,n},t_n v_n) + o(1) \ge c_\beta^0 + o(1),
\end{align*}
taking the limit, we have $c_\beta \ge c_\beta^0$. On the other hand, we have $c_\b < c_\b^0$, reaching a contradiction. Thus we have $q_0 \neq 0$. 

\textbf{Claim 3.} \quad $J(u_0,v_0) \le c_\beta$.
\\
Since 
\begin{align*}
\phi_{\lambda,n} &\rightharpoonup \phi_{\lambda,0}\ \mathrm{weakly\ in}\ H^1(\RD),\\
q_n &\to q_0\ \mathrm{in}\ \C,\\
v_n &\rightharpoonup v_0\ \mathrm{weakly\ in}\ H^1(\RD),
\end{align*}
we have
\begin{align*}
c_\beta &= \liminf_{n \to \infty} I_\beta(u_n,v_n)
 = \liminf_{n \to \infty} J(u_n,v_n)\\ 
&= \liminf_{n \to \infty} \frac{1}{4} \{\|\nabla \phi_{\lambda,n}\|_2^2 + \lambda \|\phi_{\lambda,n}\|_2^2 + (\omega - \lambda) \|u_n\|_2^2 + (\alpha + \theta_\alpha) |q_n|^2 + \|\nabla v_n\|_2^2 + \tilde{\omega} \|v_n\|_2^2\}\\
 &\ge \frac{1}{4} \{\|\nabla \phi_{\lambda,0}\|_2^2 + \lambda \|\phi_{\lambda,0}\|_2^2 + (\omega - \lambda) \|u_0\|_2^2 + (\alpha + \theta_\alpha) |q_0|^2 + \|\nabla v_0\|_2^2 + \tilde{\omega} \|v_0\|_2^2\}
= J(u_0,v_0),
\end{align*}
concluding the proof of the claim.

\textbf{Claim 4.} \quad $(u_0,v_0) \in \mathcal{N}_\beta$, namely, $G_\beta(u_0,v_0) = 0$. 
\\
We argue by contradiction. There are two alternatives.\\
\textsc{Case 1:}\quad $G_\beta(u_0,v_0) < 0$. \\
From Lemma \ref{lemm17}, we have $J(u_0,v_0) > c_\beta$. This contradicts Claim 3. So this case does not occur.\\
\textsc{Case 2:}\quad $G_\beta(u_0,v_0) > 0$.
Since
\begin{align*}
0 = G_\beta(u_n,v_n) = G_\beta(u_0,v_0) + G_\beta(u_n - u_0,v_n - v_0) + o(1),
\end{align*}
them for $n$ sufficiently large, 
\begin{align*}
G_\beta(u_n - u_0,v_n - v_0) < 0.
\end{align*}
Then from Lemma \ref{lemm17}, $J(u_n- u_0,v_n - v_0) > c_\beta$. Moreover, 
\begin{align*}
c_\beta + o(1) = J(u_n,v_n) = J(u_0,v_0) + J(u_n - u_0,v_n - v_0) + o(1).
\end{align*}
Then $J(u_0,v_0) \le 0$. Since $J(u_0,v_0) \ge 0$, we have $\phi_{\lambda,0} = v_0 = 0$ and $q_0 = 0$. This contradicts $q_0 \neq 0$. Case 2 cannot occur and this prove the claim.

\textbf{Claim 5.} \quad $\phi_{\lambda,n} \to \phi_{\lambda,0}$, $v_n \to v_0$ in $H^1(\RD)$. 
\\Set $u_0 = \phi_{\lambda,0} + q_0 \cg_\l $. Being $(u_0,v_0) \in \mathcal{N}_\beta$, we deduce that
\begin{align*}
c_\beta = \liminf_{n \to \infty} I_\beta(u_n,v_n)
 = \liminf_{n \to \infty} J(u_n,v_n) \ge J(u_0,v_0) = I_\beta(u_0,v_0) \ge c_\beta. 
\end{align*}
Therefore $J(u_n,v_n) \to J(u_0,v_0)$ and $I_\beta(u_0,v_0) = c_\beta$. Hence $(u_0,v_0)$ is a minimizer for $c_\beta$. From Lemma \ref{min-crit}, $(u_0,v_0)$ is a ground state of \eqref{eq4}. 
Moreover, since $\phi_{\lambda,n} \rightharpoonup \phi_{\lambda,0}$, $v_n \rightharpoonup v_0$ weakly in $H^1(\RD)$, we have $\phi_{\lambda,n} \to \phi_{\lambda,0}$, $v_n \to v_0$ in $H^1(\RD)$.
Moreover, by the same argument as, in the proof of Appendix A in \cite{ABCT2}, we have $\phi_{\lambda,0} \in H^2(\mathbb{R}^2)$ and $\phi_{\lambda,0}(0) = (\alpha + \theta_\lambda) q_0$. Since $q_0 \neq 0$, we have $\phi_{\lambda,0} \neq 0$. 
\end{proof}

In order to apply Theorem \ref{theo1}, we are interested under which conditions we have $c_\b<c_\b^0$.
\\
Observe that, clearly, for any $\o,\tilde{\omega}, \b$, we have
\[
c_\beta (\o,\tilde{\omega})\le c_\beta^0(\o,\tilde{\omega}).
\] 
In the following, we study more precisely the relation between $c_\b$ and $c_\b^0$. 
First, we have the following.

\begin{lemma} \label{lemm1}
%\orange{For any $\o > \omega_0$, $\tilde{\omega}>0$,}
If $c_\beta = c_\beta^0$, then for any minimizer $(u_0,v_0)\in \mathbb{H}$ for $c_\beta^0$, $(u_0,v_0)$ becomes a minimizer for $c_\beta$ and $u_0 = 0$. 
\end{lemma}

\begin{proof}
Let $(u_0,v_0) \in \mathcal{N}_\beta^0$ be a minimizer for $c_\b^0$. 
We may assume $(u_0,v_0)$ is a non-negative minimizer.
%%%
Then it follows that $(I_\beta^0)'(u_0,v_0) = 0$. 
Moreover, $(u_0,v_0)$ becomes a minimizer for $c_\beta(\omega,\tilde{\omega})$, 
%\red{I rewrote the proof in a way that did not use Lemma \ref{min-crit}.
%Because we can show that if $(I_\beta^0)'(u_0,v_0) = 0$ and $u_0 \in H^1(\RD)$, then $I_\beta'(u_0,v_0) = 0$ holds. I'M NOT SO SURE ABOUT THIS POINT AND SO I WOULD PREFER TO USE LEMMA \ref{min-crit}}
 and so, by Lemma \ref{min-crit},
%\orange{Since $u_0 \in H^1(\RD)$ and $(I_\beta^0)'(u_0,v_0) = 0$, it follows that} 
$I_\beta'(u_0,v_0) = 0$. 
Fixing $\o_\a<\l<\o$,
by the same argument as in the proof of Appendix A in \cite{ABCT2}, we have $\phi_{\lambda,0} \in H^2(\mathbb{R}^2)$ and $\phi_{\lambda,0}(0) = (\alpha + \theta_\lambda) q_0$, where $u_0 = \phi_{\lambda,0} + q_0 \cg_\l $. Being $u_0 \in H^1(\mathbb{R}^2)$, we deduce $q_0 = 0$ and so $u_0(0) = \phi_{\lambda,0}(0) = 0$. 
%This contradicts the fact that $u_0 > 0$. 
%%%
%%%
%By the same argument as in the proof of Lemma \ref{lemm2}, $(u_0,v_0)$ becomes a minimizer for $c_\b$ and $u_0(0) = 0$. 
Moreover, 
since $(I_\beta^0)'(u_0,v_0) = 0$, 
by the strong maximum principle, it follows that $u_0 = 0$. 
\end{proof}

Now the proof of Theorem \ref{th:ex} follows easily.

\begin{proof}[Proof of Theorem \ref{th:ex}]
If $c_\beta < c_\beta^0$ holds, from the proof of Theorem \ref{theo1}, $c_\beta$ has a minimizer. On the other hand, if $c_\beta = c_\beta^0$ holds, from Lemma \ref{lemm1}, $c_\beta$ has a minimizer. 
\end{proof}

\begin{lemma} \label{lemm2}
For any $\b\ge 0$, if $\omega \le \tilde{\omega}$, then it follows that $c_\beta(\o,\tilde{\omega}) < c_\beta^0(\o,\tilde{\omega})$. 
\end{lemma}

\begin{proof}
It is clear that $c_\beta(\omega,\tilde{\omega}) \le c_\beta^0(\omega,\tilde{\omega})$.  Let $(u_0,v_0) \in \mathcal{N}_\beta^0$ be a minimizer for $c_\beta^0(\omega,\tilde{\omega})$. Since $\omega \le \tilde{\omega}$, we may assume that $u_0 \neq 0$ and $u_0 \ge 0$. So, since $(I_\beta^0)'(u_0,v_0) = 0$, by strong maximum principle, we have $u_0 > 0$. 
\\
Suppose, by contradiction, that $c_\beta(\omega,\tilde{\omega}) = c_\beta^0(\omega,\tilde{\omega})$.
From Lemma \ref{lemm1}, we have $u_0 = 0$. 
This contradicts the fact that $u_0 > 0$. 
\end{proof}

Now we will remark that $c_\beta(\o,\tilde{\omega}) < c_\beta^0(\o,\tilde{\omega})$ could hold even if $\omega > \tilde{\omega}$.

%\begin{lemma} \label{lemm8}
\begin{remark} \label{lemm8}
For any $\b\ge 0$ and $\omega_* \in (\omega_\alpha,\infty)$, there exists $\eps=\eps(\b,\omega_*) > 0$ such that
for any $\o\in (\omega_*, \omega_*+\eps)$ we have
\begin{align}
\label{eq1} c_\beta(\omega,\omega_*) &< c_\beta^0(\omega,\omega_*),
\end{align}
for any $\o'\in (\omega_*-\eps, \omega_*)$ we have
\begin{align}
\label{eq2} c_\beta(\omega_*,\omega') &< c_\beta^0(\omega_*,\omega'),
\end{align}
and for any $\o\in (\omega_*, \omega_*+\eps)$ and $\o'\in (\omega_*-\eps, \omega_*)$ we have
\begin{align}
\label{eq3} c_\beta(\omega,\omega') &< c_\beta^0(\omega,\omega').
\end{align}
%where we write $c_\beta$ and $c_\beta^0$ as $c_\beta(\omega,\tilde{\omega})$ and $c_\beta^0(\omega,\tilde{\omega})$. 
%\end{lemma}
Indeed, 
since $c_\beta^0(\omega,\tilde{\omega})$ and $c_\beta(\omega,\tilde{\omega})$ are continuous with respect to $(\omega,\tilde{\omega}) \in (\omega_\alpha,\infty) \times (0,\infty)$ (cf.  \cite[Lemma 3.7]{Pomponio10}), there exists $\eps > 0$ such that
for any $\o \in (\o_*,\o_* + \eps)$, $\o' \in (\o_* - \eps,\o_*)$, we have
\begin{align*}
|c_\beta(\o,\omega_*) - c_\beta(\omega_*,\omega_*)| < \bar{\delta},\\
|c_\beta^0(\o,\omega_*) - c_\beta^0(\omega_*,\omega_*)| < \bar{\delta},\\
|c_\beta(\omega_*,\o') - c_\beta(\omega_*,\omega_*)| < \bar{\delta},\\
|c_\beta^0(\omega_*,\o') - c_\beta^0(\omega_*,\omega_*)| < \bar{\delta},\\
|c_\beta(\o,\o') - c_\beta(\omega_*,\omega_*)| < \bar{\delta},\\
|c_\beta^0(\o,\o') - c_\beta^0(\omega_*,\omega_*)| < \bar{\delta}.
\end{align*}
Then we obtain \eqref{eq1}--\eqref{eq3}. 
Here $\bar{\delta}$ is defined as follows:
\begin{align*}
\bar{\delta} := \frac{1}{3} (c_\beta^0(\omega_*,\omega_*) - c_\beta(\omega_*,\omega_*)) > 0.
\end{align*}
\end{remark}

%
%\red{Did you delete this Remark \ref{rem1}? If you thought that this Remark isn't needed, please delete this.}
%
%\blue{
%\begin{remark} \label{rem1}
%Lemma \ref{lemm8} indicates that $c_\beta(\o,\tilde{\omega}) < c_\beta^0(\o,\tilde{\omega})$ could hold even if $\omega > \tilde{\omega}$.
%\end{remark}
%}

At last, we have that $c_\beta(\o,\tilde{\omega}) < c_\beta^0(\o,\tilde{\omega})$ for $\b$ large. More precisely, recalling the definition in \eqref{b0},  we have the following.

\begin{lemma} \label{lemm3}
For any $\o > \omega_\alpha$, $\tilde{\omega}>0$, if $\beta > \beta_0(\o,\tilde{\o})$, then it follows that $c_\beta(\o,\tilde{\omega}) < c_\beta^0(\o,\tilde{\omega})$.
\end{lemma}

\begin{proof}
Let $(u_0,v_0) \in \mathcal{N}_\beta^0$ be a minimizer for $c_\beta^0(\omega,\tilde{\omega})$. Suppose that $c_\beta(\omega,\tilde{\omega}) = c_\beta^0(\omega,\tilde{\omega})$.
By Lemma \ref{min-crit}, it follows that $u_0(0) = 0$ and $(I_\beta^0)'(u_0,v_0) = 0$. Since $\beta > \beta_0$, $c_\beta^0 < c_0^0$ holds. Therefore we have $u_0 \neq 0$ and $v_0 \neq 0$. Since we may assume that $u_0 \ge 0$, by strong maximum principle, we have $u_0 > 0$ and this contradicts $u_0(0) = 0$. 
\end{proof}

\begin{corollary} \label{lemm4.5}
For any $\o > \omega_\alpha$, $\tilde{\omega}>0$, if $c_\beta = c_\beta^0$, then $\beta\le\b_0$. 
\end{corollary}

\begin{proof}
This lemma follows easily from the contraposition of Lemma \ref{lemm3}.
\end{proof}

\begin{remark}
Observe that the converse of Lemma \ref{lemm3} does not hold, namely if   $c_\beta(\o,\tilde{\omega}) < c_\beta^0(\o,\tilde{\omega})$, then we cannot conclude that $\beta > \beta_0$. Indeed by arguing as in \cite{KurataOsada21CPAA}, we have that $\b_0(\o,\tilde{\omega})>0$, for any $\o > \omega_\alpha$, $\tilde{\omega}>0$, while, at least when $\o\le  \tilde{\omega}$, by Lemma \ref{lemm2}, we know that $c_\beta(\o,\tilde{\omega}) < c_\beta^0(\o,\tilde{\omega})$ for any $\b\ge 0$.
See also, Corollary \ref{coro8}. 
\end{remark}

Now Corollary \ref{cor} is an immediate consequence of Lemma \ref{lemm2}, Lemma \ref{lemm3}
 and Theorem \ref{theo1}.

Next aim is to show that  $c_\b$ is attained by a scalar ground state for $0\le\b<\b^*$. Moreover, we investigate when the ground state becomes regular or singular.

\begin{remark} \label{rem2}
If $c_\beta = c_\beta^0$ and  $(u_0,v_0)\in \mathbb{H}_\a$  is a minimizer for $c_\beta$ with $v_0 = 0$, then it is not a minimizer for $c_\b^0$, because it would be in contradiction with Lemma \ref{lemm1}. In particular $u_0\notin \H$.
\end{remark}

\begin{lemma} \label{lemm5.5}
If $c_\beta = d(\omega)$, then for any minimizer $u_0$ for $d(\omega)$, $(u_0,0)$ is a minimizer for $c_\beta$. 
\end{lemma}

\begin{proof}
Let $u_0$ be a minimizer for $d(\o)$. Then $(u_0,0) \in \mathcal{N}_\b$. Thus we have
\begin{align*}
d(\o) = S_\o(u_0) = I_\b(u_0,0) \ge c_\b.
\end{align*}
Since $c_\b = d(\o)$, $(u_0,0)$ becomes a minimizer for $c_\b$. 
\end{proof}

%\torange{
%\red{
%To pass from $d^0(\tilde{\omega})$ to $d^0(1)$, we consider 
%\[
%v=\frac{1}{\sqrt{\tilde{\omega}}}u \left(\frac{1}{\sqrt{\tilde{\omega}}} \cdot\right)
%\]
%
%Uniqueness by monotonicity
%
%Existence by Remark
%
%$\tilde{\omega} \in (0,\omega)$: if by contradiction $\tilde{\omega}\ge \o$, then $d^0(\tilde{\omega})\ge d^0(\omega)>d(\omega)$ 
%}
%%}

\begin{lemma} \label{lemm9}
The map $\beta \mapsto c_\beta$ is continuous on $[0,\infty)$. 
\end{lemma}

\begin{proof}
Let $\hat{\beta} \in [0,\infty)$ and let $\beta_n \in (0,\infty)$ with $\beta_n \to \hat{\beta}$ as $n \to \infty$. We show $c_{\beta_n} \to c_{\hat{\beta}}$. Let $(u_0,v_0)$ be a minimizer for $c_{\hat{\beta}}$. Let $\{t_n\}_n$ be a positive numbers sequence such that $(t_n u_0, t_n v_0) \in \mathcal{N}_{\beta_n}$. Then
\begin{align}\label{pippo}
\langle (- \Delta_\alpha + \omega) u_0,u_0 \rangle + \|\nabla v_0\|_2^2 + \tilde{\omega} \|v_0\|_2^2
 = t_n^2 \left(\|u_0\|_4^4 + \|v_0\|_4^4 + \frac{\beta_n}{2} \int_{\mathbb{R}^2} |u_0 v_0|^2\right).
\end{align}
Thus $\{t_n\}_n$ is bounded. Therefore, 
\begin{align*}
c_{\hat{\beta}} = I_{\hat{\beta}}(u_0,v_0) \ge I_{\hat{\beta}}(t_n u_0,t_n v_0)
 = I_{\beta_n}(t_n u_0,t_n v_0) + o(1) \ge c_{\beta_n} + o(1).
\end{align*}
On the other hand, let $(u_n,v_n)$ be a minimizer for $c_{\beta_n}$. Let $\{s_n\}_n$ be a positive numbers sequence such that $(s_n u_n, s_n v_n) \in \mathcal{N}_{\hat{\beta}}$. Then
\begin{align}
\label{eq5}
\langle (- \Delta_\alpha + \omega) u_n,u_n \rangle + \|\nabla v_n\|_2^2 + \tilde{\omega} \|v_n\|_2^2 = s_n^2 \left(\|u_n\|_4^4 + \|v_n\|_4^4 + \frac{\hat{\beta}}{2} \int_{\mathbb{R}^2} |u_n v_n|^2\right).
\end{align}
Since $(u_n,v_n) \in \mathcal{N}_{\beta_n}$, we have
\begin{align}
\label{eq6}
\langle (- \Delta_\alpha + \omega) u_n,u_n \rangle + \|\nabla v_n\|_2^2 + \tilde{\omega} \|v_n\|_2^2 = \|u_n\|_4^4 + \|v_n\|_4^4 + \frac{\beta_n}{2} \int_{\mathbb{R}^2} |u_n v_n|^2.
\end{align}
\ \\
\textbf{Claim 1.}\quad Up to a subsequence, there exists $C > 0$ such that
\begin{align*}
\mathcal{A}_n =: \langle (- \Delta_\alpha + \omega) u_n, u_n \rangle + \|\nabla v_n\|_2^2 + \tilde{\omega} \|v_n\|_2^2 \ge C\quad (\forall n \in \mathbb{N}).
\end{align*}
From \eqref{eq6}, we have $\mathcal{A}_n \le C \mathcal{A}_n^2$. Therefore $\mathcal{A}_n \ge C$ for all $n \in \mathbb{N}$. \\
\ \\
\textbf{Claim 2.}\quad Up to a subsequence, there exists $C > 0$ such that $\|u_n\|_4^4 + \|v_n\|_4^4 \ge C$ for all $n \in \mathbb{N}$. \\
If $\|u_n\|_4^4 + \|v_n\|_4^4 \to 0$, then from \eqref{eq6}, we have
\begin{align*}
\langle (- \Delta_\alpha + \omega) u_n,u_n \rangle + \|\nabla v_n\|_2^2 + \tilde{\omega} \|v_n\|_2^2 \to 0.
\end{align*}
This contradicts Claim 1. 
\\
By \eqref{pippo},
\begin{align*}
c_{\hat{\beta}} + o(1) \ge c_{\beta_n} = I_{\beta_n}(u_n,v_n) = \frac{1}{4} \mathcal{A}_n.
\end{align*}
Therefore $\{\mathcal{A}_n\}_n$ is bounded. 
From \eqref{eq5} and Claim 2, $\{s_n\}_n$ is bounded.
Hence we obtain
\begin{align*}
c_{\beta_n} = I_{\beta_n}(u_n,v_n) \ge I_{\beta_n}(s_n u_n,s_n v_n)
 = I_{\hat{\beta}}(s_n u_n,s_n v_n) + o(1) \ge c_{\hat{\beta}} + o(1).
\end{align*}
Therefore we obtain $c_{\beta_n} = c_{\hat{\beta}} + o(1)$ as $n \to \infty$. 
\end{proof}

\begin{lemma} \label{lemm10}
If $0 \le \beta_1 < \beta_2$, then $c_{\beta_1} \ge c_{\beta_2}$. 
\end{lemma}

\begin{proof}
Let $(u_0,v_0)$ be a minimizer for $c_{\beta_1}$. Let $t_0 > 0$ be a number such that $(t_0 u_0, t_0 v_0) \in \mathcal{N}_{\beta_2}$. Then
\begin{align*}
c_{\beta_1} = I_{\beta_1}(u_0,v_0) \ge I_{\beta_1}(t_0 u_0,t_0 v_0) 
\ge I_{\beta_2}(t_0 u_0,t_0 v_0) \ge c_{\beta_2}. 
\end{align*}
\end{proof}

Recalling the definition of $\tilde{c}_\infty$ in \eqref{c_inf} and arguing as in the proof of Theorem \ref{theo1}, we have the following.

\begin{lemma}
There exists a minimizer for $\tilde{c}_\infty$.
\end{lemma}

The following lemma is needed to define $\b^*$ and 
useful for the asymptotic behaviour as $\b \to \infty$ in Section \ref{b^*:to:inf}.

\begin{lemma} \label{lemm11}
It follows that $c_\beta < \tilde{c}_\infty / \beta$ for all $\beta > 0$.
\end{lemma}

\begin{proof}
Let $(u_\infty,v_\infty)$ be a minimizer for $\tilde{c}_\infty$. Set $u_\beta := u_\infty / \sqrt{\beta}$ and $v_\beta := v_\infty / \sqrt{\beta}$. Let $t_\beta > 0$ be a number such that $(t_\beta u_\beta, t_\beta v_\beta) \in \mathcal{N}_\beta$. Then
\begin{align*}
\tilde{c}_\infty&=\tilde{I}_\infty(u_\infty, v_\infty) \ge \tilde{I}_\infty(t_\b u_\infty,t_\b  v_\infty)
\\
&= \beta \left[ \frac{t_\b^2}{2} \left(\langle (- \Delta_\alpha + \omega)  u_\b, u_\b \rangle + \|\nabla  v_\beta\|_2^2 + \tilde{\omega} \| v_\beta\|_2^2 \right)- \frac{\beta t_\b^{4}}{2} \int_{\mathbb{R}^2} |u_\beta v_\beta|^2
\right]\\
&> \beta I_\beta(t_\beta u_\beta, t_\beta v_\beta) \ge \beta c_\beta, \qquad \text{for all}\ \beta > 0.
\end{align*}
\end{proof}

In the spirit of \cite{KurataOsada21CPAA}, by 
Lemma \ref{lemm9}, Lemma \ref{lemm10} and Lemma \ref{lemm11}, 
we have the following lemma. 
\begin{lemma}\label{le:b*}
There exists the threshold $\beta^* \ge 0$:
\begin{align*}
\beta^* := \max \{\beta \ge 0 \mid c_\beta = c_0\}.
\end{align*}
\end{lemma}

%The following asymptotic behaviour will be useful in the sequel.
The following asymptotic behaviour is needed to show that $\b^* > 0$.

\begin{lemma} \label{lemm12}
Let $\beta_n \to 0$ and let $(u_n,v_n)$ be a minimizer for $c_{\beta_n}$. Then either $u_n \to 0$ in $H^1_\alpha(\mathbb{R}^2)$ or $v_n \to 0$ in $H^1(\mathbb{R}^2)$ holds.
\end{lemma}

\begin{proof}
Suppose, by contradiction, that $u_n \not \to 0$ in $H^1_\alpha(\mathbb{R}^2)$ and $v_n \not \to 0$ in $H^1(\mathbb{R}^2)$. Therefore up to a subsequence, there exists $C > 0$ such that
\begin{align*}
\|u_n\|_{H^1_{\alpha,\o}} \ge C\ \text{and}\ \|v_n\|_{H^1} \ge C\quad \text{for all}\ n \in \mathbb{N}.
\end{align*}
Let $\{s_n\}_{n}$ and $\{t_n\}_{n}$ be two positive number sequences such that $s_n u_n \in \mathcal{M}_\omega$ and $t_n v_n \in \mathcal{M}_{\tilde{\omega}}^0$. Therefore, we have
\begin{align}
\label{eq7}
&\mathcal{A}_{1,n} := \langle (- \Delta_\alpha + \omega) u_n , u_n \rangle = s_n^2 \|u_n\|_4^4 =: s_n^2 \mathcal{B}_{1,n},\\
\label{eq8}
&\mathcal{A}_{2,n} := \|\nabla v_n\|_2^2 + \tilde{\omega} \|v_n\|_2^2 = t_n^2 \|v_n\|_4^4 = :t_n^2 \mathcal{B}_{2,n}.
\end{align}
Since $(u_n,v_n)$ is a ground state of \eqref{eq4} with $\beta = \beta_n$,  we have
\begin{align*}
\mathcal{A}_{1,n} = \mathcal{B}_{1,n} + \beta_n \mathcal{C}_n,\quad 
\mathcal{A}_{2,n} = \mathcal{B}_{2,n} + \beta_n \mathcal{C}_n,
\end{align*}
where $\mathcal{C}_n := \int_{\mathbb{R}^2} |u_n v_n|^2$ and, moreover, by Lemma \ref{lemm10},
\begin{align*}
c_0 \ge c_{\beta_n} = I_{\beta_n}(u_n,v_n) = \frac{1}{4} (\mathcal{A}_{1,n} + \mathcal{A}_{2,n}).
\end{align*}
Therefore $\{\mathcal{A}_{1,n}\}_n$ and $\{\mathcal{A}_{2,n}\}_n$ are bounded. By the Sobolev embedding $H^1(\mathbb{R}^2) \hookrightarrow L^4(\mathbb{R}^2)$, there exists $C_2 > 0$ such that $\mathcal{B}_{1,n}, \mathcal{B}_{2,n}, \mathcal{C}_n \le C_2$ for all $n \in \mathbb{N}$. From the assumption of the contradiction, there exists $C_1 > 0$ such that $\mathcal{B}_{1,n} \ge C_1$ and $\mathcal{B}_{2,n} \ge C_1$. 
\\
There are two possibilities.
\ \\
\textsc{Case 1.}\quad Up to a subsequence, either $s_n \le 1$, for all $n \in \mathbb{N}$, or $t_n \le 1$, for all $n \in \mathbb{N}$ holds. \\
If $s_n \le 1$, for all $n \in \mathbb{N}$, holds, since $s_n u_n \in \mathcal{M}_\omega$, we have
\begin{align*}
c_0 \le d(\omega) \le S_\omega(s_n u_n) = \frac{s_n^4}{4} \mathcal{B}_{1,n} \le \frac{1}{4} \mathcal{B}_{1,n}. 
\end{align*}
Moreover, since $(u_n,v_n) \in \mathcal{N}_{\beta_n}$, we have
\begin{align*}
c_0 + C_1 \le \frac{1}{4} (\mathcal{B}_{1,n} +\mathcal{B}_{2,n}) 
= S_\omega(u_n) + S_{\tilde{\omega}}^0(v_n) - \beta_n \mathcal{C}_n
 \le I_{\beta_n}(u_n,v_n) = c_{\beta_n}.
\end{align*}
From Lemma \ref{lemm9}, we have $c_{\beta_n} \to c_0$. This is a contradiction.
\\
Instead, if  $t_n \le 1$, for all $n \in \mathbb{N}$, the arguments are similar.
\ \\
\textsc{Case 2.} Up to a subsequence, $s_n > 1\ (\text{for all}\ n \in \mathbb{N})$ and $t_n > 1\ (\text{for all}\ n \in \mathbb{N})$.
\ \\
Since $(u_n,v_n) \in \mathcal{N}_{\beta_n}$, we have
\begin{align*}
\mathcal{A}_{1,n} + \mathcal{A}_{2,n} = \mathcal{B}_{1,n} + \mathcal{B}_{2,n} + 2 \beta_n \mathcal{C}_n. 
\end{align*}
From \eqref{eq7} and \eqref{eq8}, it follows that
\begin{align*}
(s_n^2 - 1) \mathcal{B}_{1,n} + (t_n^2 - 1) \mathcal{B}_{2,n} = 2 \beta_n \mathcal{C}_n.
\end{align*}
Therefore $s_n = 1 + o(1)$ and $t_n = 1 + o(1)$ hold. Then
\begin{align*}
2 c_0 \le S_\omega(s_n u_n) + S_{\tilde{\omega}}^0(t_n v_n) = I_{\beta_n}(u_n,v_n) + o(1) = c_0 + o(1). 
\end{align*}
This is a contradiction.
\end{proof}

\begin{lemma} \label{lemm13}
We have that $\beta^* > 0$. 
\end{lemma}

\begin{proof}
Assume that $\beta^* = 0$. From Theorem \ref{th:ex},  there exists $\beta_n > 0$ and a minimizer $(u_n,v_n)$ for $c_{\beta_n}$ such that $\beta_n \to 0$ and  $u_n \neq 0$ and $v_n \neq 0$, since $c_{\beta_n}<c_0$, for any $n\in \N$.  Moreover $\{(u_n,v_n)\}_n$ is bounded in $\mathbb{H}_\a$.
From Lemma \ref{lemm12}, it follows that $u_n \to 0$ in $H^1_\alpha(\mathbb{R}^2)$ or $v_n \to 0$ in $H^1(\mathbb{R}^2)$. We consider the case $u_n \to 0$ in $H^1_\alpha(\mathbb{R}^2)$. 
Since $(u_n,v_n)$ is a ground state of \eqref{eq4} with $\beta = \beta_n$, we have
\begin{align*}
\mathcal{A}_n := \langle (- \Delta_\alpha + \omega)u_n,u_n \rangle &= \|u_n\|_4^4 + \beta_n \int_{\mathbb{R}^2} |u_n v_n|^2 \le C (\|u_n\|_{H^1}^4 + \beta_n \|u_n\|_{H^1}^2 \|v_n\|_{H^1}^2)\\
 &\le C (\mathcal{A}_n^2 + \beta_n \mathcal{A}_n). 
\end{align*}
Thus we have $(1 - C \beta_n) \mathcal{A}_n \le C \mathcal{A}_n^2$. We may assume $C \beta_n < 1/2$ for $n$ sufficiently large. Therefore
\begin{align*}
\mathcal{A}_n \ge \frac{1}{2C}\quad \text{for}\ n\ \text{sufficiently large}.
\end{align*}
This contradicts $\mathcal{A}_n \to 0$. 
\\
A similar argument works for the case $v_n \to 0$ in $\H$.
\end{proof}

Keeping in mind that $\cg_\omega$ is positive, the following holds. 
\begin{lemma} \label{lemm18}
Let $(u_0,v_0)$ be a minimizer for $c_\beta$ and, taking $\lambda = \omega$, $u_0 = \phi_{\omega,0} + q_0 \cg_\omega$. Then $(w_0,|v_0|)$ is also a minimizer for $c_\beta$, where $w_0 = |\phi_{\omega,0}| + |q_0| \cg_\omega$. 
\end{lemma}

\begin{proof}
%\orange{From the introduction in Fukaya-Georgiev-Ikeda \cite{FGI22}
From Lemma \ref{lemm16}, there unique exists $t_0 > 0$ such that $(t_0 w_0, t_0 |v_0|) \in \mathcal{N}_\beta$. Since $|u_0| \le w_0$, it follows that
\begin{align*}
\|\nabla \phi_{\omega,0}\|_2^2 \ge \|\nabla |\phi_{\omega,0}|\|_2^2,\quad
\|\nabla v_0\|_2^2 \ge \|\nabla |v_0|\|_2^2,\quad
\|u_0\|_4^4 \le \|w_0\|_4^4,\quad
\int_{\mathbb{R}^2} |u_0 v_0|^2 \le \int_{\mathbb{R}^2} |w_0 v_0|^2.
\end{align*}
Therefore
\begin{align*}
c_\beta = I_\beta(u_0,v_0) \ge I_\beta(t_0 u_0, t_0 v_0)
 \ge I_\beta(t_0 w_0, t_0 |v_0|) \ge c_\beta.
\end{align*}
From the uniqueness of $t_0$, we have $t_0 = 1$ and $(w_0,|v_0|) \in \mathcal{N}_\beta$, hence $(w_0,|v_0|)$ is also a minimizer for $c_\beta$. 
\end{proof}

\begin{lemma} \label{lemm19}
Let $(u_0,v_0)$ be a vector minimizer for $c_\beta$. Then $\int_{\mathbb{R}^2} |u_0 v_0|^2 > 0$ holds. 
\end{lemma}

\begin{proof}
Fix $\lambda = \omega$. Let $u_0 = \phi_{\omega,0} + q_0 \cg_\omega$. 
By multiplying $u_0$ by some constant $e^{i \theta}\ (\theta \in \mathbb{R})$, without loss of generality, we may assume $q_0 \ge 0$. 
\\
By Lemma \ref{lemm18}, $(w_0,|v_0|)$ is also a minimizer for $c_\beta$, where $w_0 = |\phi_{\omega,0}| + q_0 \cg_\omega$. 
%
%\noindent
%\textsc{Case 1:}\quad $\phi_{\omega,0} \neq 0$ and $q_0 \neq 0$.
%
\noindent
By Lemma \ref{min-crit}, it follows that $I_\beta'(w_0,|v_0|) = 0$. Therefore 
\begin{align}
\label{eq30}
\langle \nabla \chi_\o, \nabla |\phi_{\omega,0}| \rangle + \omega \langle \chi_\o, |\phi_{\omega,0}| \rangle + \overline{\xi} q_0 (\alpha + \theta_\omega) = \lr{\chi,w_0^3} + \lr{\chi,\b w_0 |v_0|^2}
\end{align}
for any $\chi = \chi_\omega + \xi \cg_\omega \in H^1_\alpha(\mathbb{R}^2)$. 
By the same argument as in Lemma \ref{lemm1}, we have $|\phi_{\omega,0}(0)| = (\alpha + \theta_\omega) q_0$.
If $\phi_{\omega,0} = 0$, then $q_0 = 0$, hence, $u_0 = 0$. This is a contradiction. Therefore $\phi_{\omega,0} \neq 0$.
Set $\xi = 0$ in \eqref{eq30}. For any $\chi_\omega \in H^1(\mathbb{R}^2)$ with $\chi_\omega \ge 0$, we have
\begin{align*}
\langle \nabla \chi_\omega, \nabla |\phi_{\omega,0}| \rangle + \omega \langle \chi_\omega, |\phi_{\omega,0}| \rangle = \lr{\chi_\o,w_0^3} + \lr{\chi_\o,\b w_0 |v_0|^2} \ge 0.
\end{align*}
By the strong maximum principle, $|\phi_{\omega,0}| > 0$, hence $q_0 > 0$. Similarly, we have $|v_0| > 0$. 
Since $I_\beta(u_0,v_0) = I_\beta(w_0,|v_0|)$, we have that $J(u_0,v_0) = J(w_0,|v_0|)$, too, and so $\|\nabla \phi_{\omega,0}\|_2^2 = \|\nabla |\phi_{\omega,0}|\|_2^2$, $\|\nabla v_0\|_2^2 = \|\nabla |v_0|\|_2^2$. 
%From Lieb-Loss \textcolor{green}{(後で引用箇所を確認する)}, 
%\\
%\red{why below? reference? As in paper of Colin-Jeanjean-Squassina? or $\t_i=\t_i(x)$ as in \cite{ABCT2}?}
%
%\red{$\to$ 
%My supervisor when I was in the PhD program, Professor Kazuhiro Kurata, taught me about this. See Claim 1 in page 2 of the file named ``note-20200526.pdf" in the dropbox.}
\\
Then there exist $\theta_1, \theta_2 \in \mathbb{R}$ and some positive functions $\Phi_{\omega,0}, \Psi_0 \in H^1(\mathbb{R}^2)$ such that
\begin{align*}
\phi_{\omega,0} = e^{i \theta_1} \Phi_{\omega,0},\quad v_0 = e^{i \theta_2} \Psi_0.
\end{align*} 
Moreover, since $\int_{\mathbb{R}^2} |u_0 v_0|^2 = \int_{\mathbb{R}^2} |w_0 v_0|^2$ and $|u_0| \le w_0$, we have $|u_0|^2 = w_0^2$, namely, 
\begin{align*}
|e^{i \theta_1} \Phi_{\omega,0} + q_0 \cg_\omega|^2 = |\Phi_{\omega,0} + q_0 \cg_\omega|^2.
\end{align*}
Hence, we have $e^{i \theta_1}= 1$. Therefore $\phi_{\omega,0} = \Phi_{\omega,0} > 0$. Thus we have $|u_0| = \phi_{\omega,0} + q_0 \cg_\omega > 0$. 
Therefore $\int_{\mathbb{R}^2} |u_0 v_0| > 0$. 
%\ \\
%\noindent
%\textsc{Case 2:}\quad $\phi_{\omega,0} = 0$.
%
%\noindent
%In this case, $u_0 = q_0 \cg_\omega > 0$. By the same argument as in \textsc{Case 1}, we have $|v_0| > 0$. 
%Therefore $\int_{\mathbb{R}^2} |u_0 v_0| > 0$. 
%
%\ \\
%\noindent
%\textsc{Case 3:}\quad $q_0 = 0$.
%
%\noindent
%In this case, $u_0 = \phi_{\omega,0} \neq 0$.
%By the same argument as in \textsc{Case 1}, we have $|\phi_{\omega,0}| > 0$ and $|v_0| > 0$. 
%Therefore $\int_{\mathbb{R}^2} |u_0 v_0| > 0$. 
\end{proof}

\begin{lemma}\label{le:dec}
The map $\b\mapsto c_\b$ is decreasing in $[\b^*,\infty)$ namely, $c_{\b_1} > c_{\b_2}$ holds for any $\b_1,\b_2 \in [\b^*,\infty)$ with $\b_1 < \b_2$.
\end{lemma}

\begin{proof}
Let $\b_1,\b_2 \in [\b^*,\infty)$ with $\b_1 < \b_2$. If $\b_1 = \b^*$, then by the definition of $\b^*$, we have $c_{\b_1} = c_{\b^*} = c_0 > c_{\b_2}$. 
If $\b_1 > \b^*$, then $c_{\b_1} < c_0$. Therefore, all minimizers for $c_{\b_1}$ are vector. 
Let $(u_0,v_0)$
be a vector minimizer for $c_{\beta_1}$. From Lemma \ref{lemm19}, $\int_{\mathbb{R}^2} |u_0 v_0|^2 > 0$ holds. 
From Lemma \ref{lemm16}, there exists $t_0 > 0$ such that $(t_0 u_0, t_0 v_0) \in \mathcal{N}_{\beta_2}$. Therefore
\begin{align*}
c_{\beta_1} = I_{\beta_1}(u_0,v_0) \ge I_{\beta_1}(t_0 u_0, t_0 v_0) 
 > I_{\beta_2}(t_0 u_0, t_0 v_0) \ge c_{\beta_2}.
\end{align*}
\end{proof}

\begin{proof}[\textbf{Proof of Theorem \ref{theo3}}]
Assume, by contradiction, that there exists a vector minimizer $(u_0,v_0)$ for $c_\beta$. Then from Lemma \ref{lemm19}, it follows that $\int_{\RD} |u_0 v_0|^2 > 0$. 
Arguing as in the proof of Lemma \ref{le:dec}, we have $c_\beta > c_{\beta'}$ for any $\beta < \beta' < \beta^*$. This contradicts $c_{\b'}=c_0$ for
 $0 \le \beta' < \beta^*$. Therefore, all the minimizers for $c_\beta$ are scalar. 
The final part of the statement follows easily.
\end{proof}

Next, we show that $c_\b$ cannot be attained by vector and regular ground state for any $\b \in [0,\infty)$. 

\begin{proof}[\textbf{Proof of Theorem \ref{theo4}}]
If $\b=0$, \eqref{eq4} reduces to two uncoupled nonlinear Schr\"odinger equations and the $c_0=\min\{d(\o),d^0(\tilde{\o})\}$ is achieved by only scalar ground states. 
\\
Let $\b>0$. If $c_\beta < c_\beta^0$, then by Theorem \ref{theo1}, all minimizers for $c_\beta$ are singular. So in this case, there is no vector and regular minimizer for $c_\beta$. 
We consider the case $c_\beta = c_\beta^0$. 
Assume, by contradiction, that there exists a vector and regular minimizer $(u_0,v_0)$ for $c_\beta$. 
Since $(u_0,v_0)$ is regular, $(u_0,v_0)$ becomes a minimizer for $c_\beta^0$. 
By Lemma \ref{lemm1}, we have $u_0 = 0$. This contradicts the fact that $(u_0,v_0)$ is vector.
\end{proof}

To prove Theorem \ref{theo5}, we need the following lemma.

\begin{lemma} \label{lemm4}
If $c_\beta = c_\beta^0$, then $c_\beta^0 = d^0(\tilde{\omega})$. 
\end{lemma}

\begin{proof}
It is clear that $c_\b^0 \le d^0(\tilde{\o})$. 
Let $(u_0,v_0)$ be a minimizer for $c_\b^0$. 
From Lemma \ref{lemm1}, $u_0 = 0$. Then $v_0 \in \mathcal{M}_{\tilde{\o}}^0$. Therefore
\begin{align*}
c_\b^0 = I_\b^0(u_0,v_0) = S_{\tilde{\o}}^0(v_0) \ge d^0(\tilde{\o}).
\end{align*}
\end{proof}

\begin{remark}
If $\o \le \tilde{\omega}$, the assumptions of Lemma \ref{lemm4} cannot be satisfied.
Because if $\omega \le \tilde{\omega}$, then by Lemma \ref{lemm2}, we have $c_\beta(\omega,\tilde{\omega}) < c_\beta^0(\omega,\tilde{\omega})$. 
\end{remark}

\begin{proof}[\textbf{Proof of Theorem \ref{theo5}}]
By Lemma \ref{lemm2} we immediately have $\omega > \tilde{\omega}$.
From Lemma \ref{lemm4}, $c_\beta^0 = d^0(\tilde{\omega})$ holds. 
Assume, by contradiction, that  $c_\beta < c_0$. Then we have $c_\beta < c_0 = \min \{d(\omega), d^0(\tilde{\omega})\} \le d^0(\tilde{\omega})$. 
This contradicts the fact that $c_\beta = c_\beta^0 = d^0(\tilde{\omega})$. 
Therefore $c_\b^0=c_\beta = c_0=d^0(\tilde{\omega})$. The fact that $c_0 = d^0(\tilde{\o})$ implies $d(\o) \ge d^0(\tilde{\o})$. Being   $\omega > \tilde{\omega}$ and $ c_0^0 = \min \{d^0(\omega), d^0(\tilde{\omega})\}=d^0(\tilde{\omega})$, we deduce the first part of the statement. 
\\
Since $c_{\beta^*} = c_0$, $c_{\beta_0}^0 = c_0^0$ and $c_\beta$ and $c_\beta^0$ are non-increasing on $[0,\infty)$, we have 
\begin{align}
&c_{\beta'} = c_0 = d^0(\tilde{\omega})\quad \text{for any}\ \beta' \in [0,\beta^*],\notag\\
\label{eq14}
&c_{\beta'}^0 = c_0^0 = d^0(\tilde{\omega})\quad \text{for any}\ \beta' \in [0,\beta_0]. 
\end{align}
Therefore we have 
$c_{\beta'} = c_{\beta'}^0\ (\text{for any}\ \beta' \in [0,\min\{\beta^*,\beta_0\}])$. 
It follows that 
\begin{align*}
&c_{\beta'} = c_0 = d^0(\tilde{\omega})\quad \text{for any}\ \beta' \in [0,\beta^*],\\
&c_{\beta'}^0 < c_0^0 = d^0(\tilde{\omega})\quad \text{for any}\ \beta' \in (\beta_0,\infty).
\end{align*}
We assume $\beta^* > \beta_0$. If we take $\beta' \in (\beta_0,\beta^*)$, then $d^0(\tilde{\omega}) = c_{\beta'} \le c_{\beta'}^0 < d^0(\tilde{\omega})$. This is a contradiction. Therefore we have $\beta^* \le \beta_0$. 
Thus we have 
\begin{align*}
c_{\beta'} = c_{\beta'}^0\quad \text{for any}\ \beta' \in [0,\beta^*]. 
\end{align*}
By \eqref{eq14} and since
\begin{align*}
c_{\beta'} < c_0 = d^0(\tilde{\omega})\quad \text{for any}\ \beta' \in (\beta^*,\infty),
\end{align*}
it follows that $c_{\beta'} < d^0(\tilde{\omega}) = c_{\beta'}^0\ (\text{for any}\ \beta' \in (\beta^*,\beta_0])$. Moreover, from Lemma \ref{lemm3}, we have $c_{\beta'} < c_{\beta'}^0$ for any $\beta' \in (\beta_0,\infty)$. Therefore we obtain
\begin{align*}
c_{\beta'} < c_{\beta'}^0\quad \text{for any}\ \beta' \in (\beta^*,\infty). 
\end{align*}
Hence, we have $\b \le \b^*$. 
\end{proof}

\begin{proof}[\textbf{Proof of Theorem \ref{theo7}}]

 Suppose \eqref{7i}. From Theorem \ref{theo5}, we have $d(\omega) \ge d^0(\tilde{\omega})$ and $\beta \le \beta^* (\le \beta_0)$. Therefore \eqref{7ii} holds.

\noindent
 Suppose \eqref{7ii}. Note that $d(\omega) < d^0(\omega)$. From $d(\omega) \ge d^0(\tilde{\omega})$, we have
\begin{align*}
&c_0 = \min \{d(\omega), d^0(\tilde{\omega})\} = d^0(\tilde{\omega}),\\
&c_0^0 = \min \{d^0(\omega), d^0(\tilde{\omega})\} = d^0(\tilde{\omega}).
\end{align*}
Thus $c_0 = c_0^0$ holds. 
Applying Theorem \ref{theo5} as $\beta = 0$, we obtain $\beta^* \le \beta_0$. 
From the assumption of \eqref{7ii}, $0 \le \beta \le \beta^*$ holds. Thus $0 \le \beta \le \beta^*$ and $0 \le \beta \le \beta_0$ hold. Hence $c_\beta = c_0$ and $c_\beta^0 = c_0^0$ hold. 
Since $c_0 = c_0^0$ holds, we obtain $c_\beta = c_\beta^0$. 

\end{proof}

Finally we conclude this section showing that  $c_\b$ is attained by a singular and vector ground state for $\b>\b^*$.

\begin{proof}[\textbf{Proof of Theorem \ref{theo6}}]
By the definition of $\b^*$ in \eqref{b*}, we deduce that  $c_\beta < c_0$ and so the ground state is vector. Moreover, 
from Remark \ref{remm}, being $\beta > \beta^*$, 
we have $c_\beta < c_\beta^0$. Therefore, by Theorem \ref{theo1}, we conclude that any minimizer for $c_\b$ is vector and singular.
\end{proof}

\section{Asymptotic behaviour as $\beta \to \infty$}
\label{b^*:to:inf}

In this last section, we study the asymptotic behaviour of ground state solutions of \eqref{eq4}, as $\b\to \infty$.

Let $(u,v)$ be a solution of \eqref{eq4}. 
Set $(w,z) = (\sqrt{\b} u, \sqrt{\b} v)$. Then $(w,z)$ satisfies the following system:
\begin{align}
\label{eq10}
\begin{cases}
- \Delta_\alpha w + \omega w = \frac{1}{\b} |w|^2 w + w |z|^2&\quad \mathrm{in}\ \mathbb{R}^2,\\
- \Delta z + \tilde{\omega} z = \frac{1}{\b} |z|^2 z + |w|^2 z&\quad \mathrm{in}\ \mathbb{R}^2.
\end{cases}
\tag{$\mathcal{\tilde{P}}_\beta$}
\end{align}
Moreover, we define the following corresponding functional and minimization problem
\[
\tilde{c}_\beta := \inf_{(w,z) \in \mathcal{\tilde{N}}_\beta} \tilde{I}_\beta(w,z),
\]
where $\tilde{I}_\beta : \mathbb{H}_\a \to \mathbb{R}$ is the functional so defined
\begin{align*}
\tilde{I}_\beta(w,z)& := \frac{1}{2} \left(
\langle (- \Delta_\alpha + \omega)w,w \rangle + \|\nabla z\|_2^2 + \tilde{\omega} \|z\|_2^2
\right) - \frac{1}{4 \b} (\|w\|_4^4 + \|z\|_4^4) - \frac{1}{2} \int_{\mathbb{R}^2} |wz|^2\\
&= \frac{1}{2} \left(\|\nabla \phi_\lambda\|_2^2 + \lambda \|\phi_\lambda\|_2^2 + (\omega - \lambda) \|w\|_2^2 + (\alpha + \theta_\lambda) |q|^2\right)\\
&\quad \ + \frac{1}{2} \left(\|\nabla z\|_2^2 + \tilde{\omega} \|z\|_2^2\right) - \frac{1}{4 \b} (\|w\|_4^4 + \|z\|_4^4) - \frac{1}{2} \int_{\mathbb{R}^2} |wz|^2,
\end{align*}
$\mathcal{\tilde{N}}_\beta$ is its Nehari manifold
\[
\mathcal{\tilde{N}}_\beta := \big\{(w,z) \in \mathbb{H}_\a \setminus \{(0,0)\} \mid \tilde{G}_\beta(w,z) = 0\big\}
\]
with
\begin{align*}
\tilde{G}_\beta(w,z)& := \langle (- \Delta_\alpha + \omega)w,w \rangle + \|\nabla z\|_2^2 + \tilde{\omega} \|z\|_2^2 - \frac{1}{\b} (\|w\|_4^4 + \|z\|_4^4) - 2 \int_{\mathbb{R}^2} |wz|^2\\
&= \|\nabla \phi_\lambda\|_2^2 + \lambda \|\phi_\lambda\|_2^2 + (\omega - \lambda) \|w\|_2^2 + (\alpha + \theta_\lambda) |q|^2\\
&\quad \ + \|\nabla z\|_2^2 + \tilde{\omega} \|z\|_2^2 - \frac{1}{\b} (\|w\|_4^4 + \|z\|_4^4) - 2 \int_{\mathbb{R}^2} |wz|^2.
\end{align*}

%\red{by the following thm we should have that \eqref{eq11} has a vector and singular GS. we know that any GS of \eqref{eq11} is vector. is it true that any GS  \eqref{eq11} is also singular?
%is it worthy to think about this?}
%
%\red{$\to$ Yes. You are right. any GS of \eqref{eq11} is singular because $\tilde{c}_\infty < \tilde{c}_\infty^0$. I don't know whether it is worthy to think about this or not. But
%I don't think we need to include that in this paper at this time.
%If we eventually realize its value, then put it in our paper then.}

\begin{proof}[\textbf{Proof of Theorem \ref{theo2}}]
Let $(w_n,z_n) = (\sqrt{\beta_n} u_n, \sqrt{\beta_n} v_n)$. We can easily find that $I_{\beta_n}(u_n,v_n) = \frac{1}{\beta_n} \tilde{I}_{\beta_n}(w_n,z_n)$ and $c_{\beta_n} = \tilde{c}_{\beta_n} / \beta_n$. Then $(w_n,z_n)$ is a ground state of $(\mathcal{\tilde{P}}_{\beta_n})$.
Let $\psi_{\lambda,n}:=\sqrt{\beta_n} \phi_{\lambda,n} $ and $r_n := \sqrt{\beta_n} q_n$ so that  $w_n = \psi_{\lambda,n} + r_n \cg_\l $.  
From Lemma \ref{lemm11} and 
\begin{align*}
\tilde{c}_\infty &\ge \tilde{c}_{\beta_n} = \tilde{I}_{\beta_n}(w_n,z_n)\\
 &= \frac{1}{4} \langle (- \Delta_\alpha + \omega) w_n,w_n \rangle + \|\nabla z_n\|_2^2 + \tilde{\omega} \|z_n\|_2^2,
\end{align*}
$\{\psi_{\lambda,n}\}_{n}$, $\{z_n\}_{n}$ is bounded in $H^1(\mathbb{R}^2)$ and $\{r_n\}_{n}$ is bounded in $\mathbb{C}$. Up to a subsequence, there exist $\psi_{\lambda,0}, z_0 \in H^1(\mathbb{R}^2)$ and $r_0 \in \mathbb{C}$ such that
\begin{align*}
&\psi_{\lambda,n} \rightharpoonup \psi_{\lambda,0}\quad \mathrm{weakly\ in}\ H^1(\mathbb{R}^2),\\
&r_n \to r_0\quad \mathrm{in}\ \mathbb{C},\\
&z_n \rightharpoonup z_0\quad \mathrm{weakly\ in}\ H^1(\mathbb{R}^2).
\end{align*}
Since $(w_n,z_n)$ satisfies $(\mathcal{\tilde{P}}_{\beta_n})$, $(w_0,z_0)$ satisfies \eqref{eq11}, where  $w_0 = \psi_{\lambda,0} + r_0 \cg_\l $. 
\ \\
\noindent
\textbf{Claim.}\ $r_0 \neq 0$. \\
Assume $r_0 = 0$. Then 
\begin{align}
\label{eq13}
\begin{aligned}
&\|w_n - \psi_{\lambda,n}\|_p \to 0\qquad \text{for }p \in [2,\infty),\\
&\left|\|w_n\|_4^4 - \|\psi_{\lambda,n}\|_4^4\right| \to 0,\\
&\left|\int_{\mathbb{R}^2} |w_n z_n|^2 - \int_{\mathbb{R}^2} |\psi_{\lambda,n} z_n|^2\right| \to 0. 
\end{aligned}
\end{align}
It's standard to prove the existence of a unique positive number $t_n$ 
such that $(t_n \psi_{\lambda,n},t_n z_n) \in \mathcal{\tilde{N}}_\infty^0$, where
\begin{align*}
&\tilde{c}_\infty^0 := \inf_{(u,v) \in \mathcal{\tilde{N}}_\infty^0} \tilde{I}_\infty^0(u,v),\\
&\tilde{I}_\infty^0 : \mathbb{H} \to \mathbb{R},\\
&\tilde{I}_\infty^0(u,v) := \frac{1}{2} \left( \|\nabla u\|_2^2 + \omega \|u\|_2^2 + \|\nabla v\|_2^2 + \tilde{\omega} \|v\|_2^2\right) - \frac{1}{2} \int_{\mathbb{R}^2} |uv|^2,\\
&\mathcal{\tilde{N}}_\infty^0 := \big\{(u,v) \in \mathbb{H} \setminus \{(0,0)\} \mid \tilde{G}_\infty^0(u,v) = 0\big\},\\
&\tilde{G}_\infty^0(u,v) := 
\|\nabla u\|_2^2 + \omega \|u\|_2^2 + \|\nabla v\|_2^2 + \tilde{\omega} \|v\|_2^2 - 2 \int_{\mathbb{R}^2} |uv|^2.
\end{align*}
Thus
\begin{align}
\label{eq12}
\|(\psi_{\lambda,n},z_n)\|_{\mathbb{H}}^2 = 2 t_n^2 \int_{\mathbb{R}^2} |\psi_{\lambda,n} z_n|^2. 
\end{align}
Now we prove $\{t_n\}_n$ is bounded. Since $(w_n,z_n) \in \mathcal{\tilde{N}}_{\beta_n}$, we have
\begin{align*}
\|(w_n,z_n)\|_{\mathbb{H}_\alpha}^2 &= \frac{1}{\beta_n} (\|w_n\|_4^4 + \|z_n\|_4^4) + 2 \int_{\mathbb{R}^2} |w_n z_n|^2
%\\ &
\le C \|(w_n,z_n)\|_{\mathbb{H}_\alpha}^4
\end{align*}
Therefore we have
\begin{align*}
\|(w_n,z_n)\|_{\mathbb{H}_\alpha}^2 \ge C\quad \mathrm{for\ all}\ n \in \mathbb{N}.
\end{align*}
%\orange{Moreover, 
%\begin{align*}
%\|(w_n,z_n)\|_{\mathbb{H}_\alpha}^2 &\le \frac{C}{\beta_n} \|(w_n,z_n)\|_{\mathbb{H}_\alpha}^4 + 2 \int_{\mathbb{R}^2} |w_n z_n|^2\\
%&\le \frac{C}{\beta_n} \|(w_n,z_n)\|_{\mathbb{H}_\alpha}^2 + 2 \int_{\mathbb{R}^2} |w_n z_n|^2. 
%\end{align*}
%Therefore
%\begin{align*}
%C \le \left(1 - \frac{C}{\beta_n}\right) \|(w_n,z_n)\|_{\mathbb{H}_\alpha}^2 \le 2 \int_{\mathbb{R}^2} |w_n z_n|^2\qquad \text{for $n$ sufficiently large}. 
%\end{align*}}
This, together with the boundedness of $\{(w_n,z_n)\}_n$ in $\mathbb{H}_\a$ and since $\{\b_n\}_n$ is a diverging sequence, implies that 
\[
C \le  \|(w_n,z_n)\|_{\mathbb{H}_\alpha}^2 
-\frac{1}{\beta_n} (\|w_n\|_4^4 + \|z_n\|_4^4) \le 2 \int_{\mathbb{R}^2} |w_n z_n|^2\qquad \text{for $n$ sufficiently large}. 
\]
From \eqref{eq13} and \eqref{eq12}, $\{t_n\}_n$ is bounded. Therefore,  by Lemma \ref{lemm11} and the relation between $c_{\beta_n} $ and $\tilde{c}_{\beta_n} $,
\begin{align*}
\tilde{c}_\infty &\ge \tilde{c}_{\beta_n} = \tilde{I}_{\beta_n}(w_n,z_n)
 \ge \tilde{I}_{\beta_n}(t_n w_n, t_n z_n)\\
 &= \tilde{I}_{\beta_n}^0(t_n \psi_{\lambda,n}, t_n z_n) + o(1)
 = \tilde{I}_\infty^0(t_n \psi_{\lambda,n}, t_n z_n) + o(1)
 \ge \tilde{c}_\infty^0 + o(1). 
\end{align*}
Hence we have $\tilde{c}_\infty \ge \tilde{c}_\infty^0$. By the same argument as in the proof of Lemma \ref{lemm2} and Lemma \ref{lemm3}, we can prove $\tilde{c}_\infty < \tilde{c}_\infty^0$ (see also  \cite{AS05} for the existence of a minimizer for $\tilde{c}_\infty^0$). This is a contradiction. 

\medskip
\noindent
Thus $r_0 \neq 0$ and $(w_0,z_0)$ is a nontrivial solution of \eqref{eq11}. 
Hence, $(w_0,z_0) \in \mathcal{\tilde{N}}_\infty$.  Therefore 
\begin{align*}
\tilde{c}_\infty &\ge \liminf_{n \to \infty} \tilde{c}_{\beta_n} = \liminf_{n \to \infty} \tilde{I}_{\beta_n}(w_n,z_n)\\
&= \frac{1}{4} \liminf_{n \to \infty} (\|\nabla \psi_{\lambda,n}\|_2^2 + \lambda \|\psi_{\lambda,n}\|_2^2 + (\omega - \lambda) \|\psi_{\lambda,n}\|_2^2 + (\alpha + \theta_\lambda) |r_n|^2 + \|\nabla z_n\|_2^2 + \tilde{\omega} \|z_n\|_2^2)\\
&\ge \frac{1}{4} (\|\nabla \psi_{\lambda,0}\|_2^2 + \lambda \|\psi_{\lambda,0}\|_2^2 + (\omega - \lambda) \|\psi_{\lambda,0}\|_2^2 + (\alpha + \theta_\lambda) |r_0|^2 + \|\nabla z_0\|_2^2 + \tilde{\omega} \|z_0\|_2^2)\\
&= \tilde{I}_\infty(w_0,z_0) \ge \tilde{c}_\infty. 
\end{align*}
We obtain $\tilde{c}_\beta \to \tilde{c}_\infty\ (\mathrm{as}\ \beta \to \infty)$ and $\tilde{I}_\infty(w_0,z_0) = \tilde{c}_\infty$. Hence $(w_0,z_0)$ is a minimizer for $\tilde{c}_\infty$. 
By the same argument as in Lemma \ref{min-crit}, $(w_0,z_0)$ is a ground state of \eqref{eq11}. 
Moreover, since $\psi_{\lambda,n} \rightharpoonup \psi_{\lambda,0}$, $z_n \rightharpoonup z_0$ weakly in $H^1(\mathbb{R}^2)$, we obtain
\begin{align*}
&\psi_{\lambda,n} \to \psi_{\lambda,0}\quad \mathrm{in}\ H^1(\mathbb{R}^2),\\
&z_n \to z_0\quad \mathrm{in}\ H^1(\mathbb{R}^2).
\end{align*}
\end{proof}

\section*{Acknowledgment}
The authors wish to thank Professors  Riccardo Adami, Raffaele Carloni, Gustavo de Paula Ramos, and Tatsuya Watanabe for useful discussions about the problem.

%============================================================
% References

%============================================================

\end{document}